\documentclass[12pt]{article}
\usepackage{amsthm}
\usepackage{amsfonts}
\usepackage{amsmath}
\usepackage{graphicx}%
\usepackage{amssymb}
\usepackage{hyperref}

\newtheorem{theorem}{Theorem}
\newtheorem{proposition}[theorem]{Proposition}
\newtheorem{lemma}[theorem]{Lemma}
\newtheorem{corollary}[theorem]{Corollary}
\DeclareMathOperator{\crg}{cr}
\DeclareMathOperator{\lcr}{\overline{cr}}
\usepackage{color}

\definecolor{olivegreen}{rgb}{0.05,0.7,0.1}

\begin{document}

\title{The 2-page crossing number of $K_{n}$}
\author{Bernardo M.~\'{A}brego\\{\small California State University, Northridge}\\{\small bernardo.abrego@csun.edu }
\and Oswin Aichholzer\\{\small Graz University of Technology}\\{\small oaich@ist.tugraz.at}
\and Silvia Fern\'{a}ndez-Merchant\\{\small California State University, Northridge}\\{\small silvia.fernandez@csun.edu }
\and Pedro Ramos\\{\small Universidad de Alcal\'{a}}\\{\small pedro.ramos@uah.es}
\and Gelasio Salazar\\{\small Universidad Aut\'{o}noma de San Luis Potos\'{\i}}\\{\small gsalazar@ifisica.uaslp.mx}}
\maketitle

\begin{abstract}
Around 1958, Hill described how to draw the complete graph $K_n$ with
\[
Z\left(  n\right) :=\frac{1}{4}\left\lfloor \frac{n}{2}\right\rfloor
\left\lfloor \frac{n-1}{2}\right\rfloor \left\lfloor \frac{n-2}{2}%
\right\rfloor \left\lfloor \frac{n-3}{2}\right\rfloor
\]
crossings, and conjectured that the crossing number $\crg (K_{n})$
of $K_n$ is exactly $Z(n)$. This is also known as Guy's
conjecture as he later popularized it. Towards the end of the
century, substantially different drawings of $K_{n}$ with $Z(n)$
crossings were found. These drawings are \emph{2-page book
drawings}, that is, drawings where all the vertices are on a line
$\ell$ (the spine) and each edge is fully contained in one of the
two half-planes (pages) defined by~$\ell$. The \emph{2-page crossing
number }of $K_{n} $, denoted by $\nu_{2}(K_{n})$, is
the minimum number of crossings determined by a 2-page book drawing of $K_{n}%
$. Since $\crg(K_{n}) \le\nu_{2}(K_{n})$ and $\nu_{2}(K_{n}) \le Z(n)$, a
natural step towards Hill's Conjecture is the
weaker conjecture $\nu_{2}(K_{n}) = Z(n)$, popularized by Vrt'o.
In this paper we develop a novel and innovative technique to investigate
crossings in drawings of $K_{n}$, and use it to prove that $\nu_{2}(K_{n}) =
Z(n) $. To this end, we extend the inherent geometric definition of $k$-edges
for finite sets of points in the plane to topological drawings of $K_{n}$. We
also introduce the concept of ${\leq}{\leq}k$-edges as a useful generalization
of ${\leq}k$-edges and extend a powerful theorem that expresses the
number of crossings in a rectilinear drawing of $K_{n}$ in terms of its number
of $(\le k)$-edges to the topological setting.
Finally, we give a complete characterization of crossing minimal
  2-page book drawings of $K_{n}$ and show that, up to equivalence,
  they are unique for $n$ even, but that there exist an exponential
  number of non-homeomorphic such drawings for $n$ odd.
\end{abstract}


\section{Introduction}

In a \emph{drawing} of a graph in the plane, each vertex is
represented by a point and each edge is represented by a simple open
arc, such that if $uv$ is an edge, then the closure (in the
plane) of the arc $\alpha$ representing $uv$ consists
precisely of $\alpha$ and the points representing $u$ and $v$. It is
further required that no arc representing an edge contains a point
representing a vertex.

A crossing in a drawing $D$ of a graph $G$ {is a pair}
$(x,\{\alpha,\beta\})$, where $x$ is a point in the plane,
$\alpha,\beta$ are arcs representing different edges, and $x\in
\alpha\cap \beta$.
The \emph{crossing number} $\crg(D)$ of $D$ is the number of crossings in
$D$, and the \emph{crossing number} $\crg(G)$ of $G$ is the minimum $\crg(D)$,
taken over all drawings $D$ of $G$.

A drawing is \emph{good} if  (i) no three distinct arcs representing
edges meet at a common point;  (ii) if two edges are adjacent, then
the arcs representing them do not intersect each other; and  (iii) an
intersection point between two arcs representing edges is a crossing
rather than tangential. It is well-known (and easy to prove) that every graph has a
crossing-minimal drawing which is good {(moreover, (ii) and
  (iii) hold in {\em every} crossing-minimal drawing)}. Thus, when our aim (as in this paper)
is to estimate the crossing number of a graph, we may assume that all drawings
under consideration are good.

As usual, for simplicity we often
make no distinction between a vertex and the point representing it, or
between an edge and the arc representing it. No confusion should arise
from this practice.

Around 1958, Hill conjectured that
\begin{equation}
\label{eq:hillc}\crg(K_{n}) =Z(n) :=\frac{1}{4}\left\lfloor \frac{n}%
{2}\right\rfloor \left\lfloor \frac{n-1}{2}\right\rfloor \left\lfloor
\frac{n-2}{2}\right\rfloor \left\lfloor \frac{n-3}{2}\right\rfloor .
\end{equation}
This conjecture appeared in print a few years later in papers by Guy
\cite{Guy60} and Harary and Hill \cite{HH63}. Hill described drawings of
$K_{n}$ with $Z(n)$ crossings, which were later corroborated by Bla\v{z}ek and
Koman \cite{BC64}. These drawings show that $\crg(K_{n}) \le Z(n)$. The best
known general lower bound is $\lim_{n\to\infty}\crg(K_{n})/Z(n) \ge0.8594$,
due to de Klerk et al.\ \cite{deCPS07}. For more on the history of this
problem we refer the reader to the excellent survey by Beineke and Wilson
\cite{BW10}.

One of the major motivations for investigating crossing numbers is their
application to VLSI design. With this motivation in mind, Chung, Leighton and
Rosenberg~\cite{Chung87} analyzed embeddings of graphs in books: the vertices
lie on a line (the \emph{spine}) and the edges lie on the \emph{pages} of the
book. Book embeddings of graphs have been extensively studied \cite{bil,dw}.
Now if the book has $k$ pages, and crossings among edges are allowed, the
result is a $k$-\emph{page book drawing}.

Here we concentrate on $2$-page book drawings. The \emph{2-page crossing
number} $\nu_{2}(G)$ of a graph $G$ is the minimum of $\crg(D)$ taken over all
2-page book drawings $D$ of $G$. Alternative terminologies for the 2-page
crossing number are \emph{circular crossing number}~\cite{Harb02} and
\emph{fixed linear crossing number}~\cite{CimMu07}. We may regard the pages as
the closed half-planes defined by the spine, and so every $2$-page book
drawing can be realized as a plane drawing; it follows that $\crg(G) \le
\nu_{2}(G)$ for every graph $G$.

In 1964, Bla\v{z}ek and Koman\ \cite{BC64} found 2-page book
drawings of
  $K_{n}$ with $Z(n)$ crossings, thus showing that $\nu_{2}(K_{n})\le
  Z(n)$ (see also Guy et al.\ \cite{Guy68}, Damiani et al.\ \cite{DaDASa}, Harborth
  \cite{Harb02}, and Shahrokhi et al.~\cite{SSSV96}.) Once these
constructions were known, the conjecture that $\nu_{2}(K_{n}) =Z(n)$
is implicit in the conjecture given by Equation~(\ref{eq:hillc})
since $\crg(K_{n}) \leq\nu_{2}(K_{n})$.  However, the only explicit
reference to this weaker conjecture is, as far as we know, from
Vrt'o~\cite{vrto2}.

Buchheim and Zhang~\cite{buchheim} reformulated the problem of
finding $\nu_{2}(K_{n})$ as a maximum cut problem on associated
graphs, and then solved exactly this maximum cut problem for all
$n\le13$, thus confirming Equation~(\ref{eq:hillc}) for $2$-page
book drawings for all $n\le14$ (the case $n=14$ follows from the
case $n=13$ by an elementary counting argument). Very recently, De
Klerk and Pasechnik~\cite{dkp} used this max cut reformulation to
find the exact value of $\nu_{2}(K_{n})$ for all $n\le21$ and
$n=24$, and moreover, by using semidefinite programming techniques,
to obtain the asymptotic bound
$\lim_{n\to\infty}\nu_2(K_{n})/Z(n) \ge~0.9253$. All the
results reported in \cite{buchheim} and~\cite{dkp} are
computer-aided.

In this paper we prove that $\nu_{2}(K_{n})=Z(n)$. The main technique for the
proof is the extension of the concept of \emph{$k$-edge} of a finite set of
points to topological drawings of the complete graph. We do this in a way such
that the identities proved by \'{A}brego and Fern\'{a}ndez-Merchant \cite{AF05}
and Lov\'{a}sz et al.\ \cite{LVWW04}, that express the crossing number of a
rectilinear drawing of $K_{n}$ in terms of the $k$-edges or the $(\le k)$-edges of its set of
vertices, are also valid in the topological setting.

We recall that a drawing $D$ is \emph{rectilinear} if the edges of $D$ are
straight line segments, and the \emph{rectilinear crossing number} $\lcr(G)$
of a graph $G$ is the minimum of $\crg(D)$ taken over all rectilinear drawings
$D$ of $G$. An edge $pq$ of $D$ is a $k$-\emph{edge} if the line spanned by
$pq$ divides the remaining set of vertices into two subsets of cardinality $k$
and $n-2-k$. Thus a $k$-edge is also an $(n-2-k)$-edge. Denote by $E_{k}(D)$
the number of $k$-edges of $D$. The following identity \cite{AF05,LVWW04} has
been key to the recent developments on the rectilinear crossing number of
$K_{n}$.%

\begin{equation}
\label{eq:tu}\lcr\left(  D\right)  =3\binom{n}{4}-\sum\limits_{k=0}%
^{\left\lfloor n/2\right\rfloor -1}k\left(  n-2-k\right)  E_{k}\left(
D\right)  .
\end{equation}

In Section~\ref{sec:croedg} we generalize the concept of $k$-edge to
arbitrary (that is, not necessarily rectilinear) drawings of
$K_{n}$. This allows us to extend Equation~(\ref{eq:tu}) to (good)
topological drawings of $K_{n}$. The key observation to extend the
definition of $k$-edge to the new setting is to observe that,
although half-planes are not well defined, we can use the
orientation of the triangles defined by three points: the edge $pq$
will be a $k$-edge of the topological drawing if the set of
triangles adjacent to $pq$ is divided, according to its orientation,
into two subsets with cardinality $k$ and $n-k-2$. In
Section~\ref{s:2page} we use this tool to show that
$\nu_{2}(K_{n})=Z(n)$. In order to do that, we need to introduce the
new concept of ${\leq}{\leq}k$-edges, because for topological
drawings the lower bound for ${\leq}k$-edges, $E_{\leq k}(D)
\geq3\tbinom{k+2}{2}$ does not hold. In Section~\ref{s:OptConfig} we
analyze crossing optimal 2-page drawings of $K_{n}$. We give a
complete characterization of their structure, showing that, up to
equivalence (see Section~\ref{section:equivalent} for a detailed
definition), crossing optimal drawings are unique for $n$ even. In
contrast, for $n$ odd we provide a family of size $2^{(n-5)/2}$ of
non-equivalent crossing optimal drawings.  We conclude with some
open questions and directions for future research in
Section~\ref{sec:concludingremarks}.

An extended abstract of this paper \cite{AAFRS} has appeared. In it
we include some additional observations on the structure of crossing
optimal $2$-page drawings of $K_n$. For instance, in these drawings
the above mentioned inequality $E_{\leq k}\left( D\right) \geq
  3\binom{k+2}{2}$ does hold.

\section{Crossings and $k$-edges}

\label{sec:croedg} In this section we generalize the concept of $k$-edges,
which has so far only been used in the geometric setting of finite sets of
points in the plane, to topological drawings of $K_{n}$. Let $D$ be a good
drawing of $K_{n}$, let $\overrightarrow{pq}$ be a directed edge of
$D$, and $r$ a vertex of $D$ other than~$p$ or~$q$. We say that \emph{$r$ is on the
left (respectively, right) side of $\overrightarrow{pq}$} if the topological
triangle $pqr$ traced in that order (its vertices and edges correspond to
those in $D$) is oriented counterclockwise (respectively, clockwise). Note
that this is well defined as the three edges $pq$, $qr$, and $rp$ in $D$ do
not self intersect and do not intersect each other, since $D$ is good. We say
that the edge $pq$ is a \emph{$k$-edge of $D$} if it has exactly $k$ points of
$D$ on the same side (left or right), and thus $n-2-k$ points on the other
side. Hence, as in the geometric setting, a $k$-edge is also an $(n-2-k)$-edge. Note that the
direction of the edge $pq$ is no longer relevant and every edge of $D$ is a
$k$-edge for some unique $k$ such that $0 \le k \le\lfloor n/2 \rfloor-1$. Let
$E_{k}(D)$ be the number of $k$-edges of $D$.

\begin{theorem}
\label{CrvsEdges}For any good drawing $D$ of $K_{n}$ in the plane the
following identity holds,
\[
\crg\left(  D\right)  =3\binom{n}{4}-\sum\limits_{k=0}^{\left\lfloor
n/2\right\rfloor -1}k\left(  n-2-k\right)  E_{k}\left(  D\right)  .
\]

\end{theorem}

\begin{proof}
In a good drawing of $K_{n}$, we say that an edge $pq$\emph{\ separates} the
vertices $r$ and $s$ if the orientations of the triangles $pqr$ and $pqs$ are
opposite.
In this case, we say that the set $\{pq,r,s\}$ is a \emph{separation}. It is
straightforward to check that, up to ambient isotopy equivalence, there are only three
  different good drawings ${A,B,C}$ of $K_{4}$; these are shown in Figure~\ref{types}.
\begin{figure}[ptb]
\par
\begin{center}
\includegraphics[width=3.25in ] {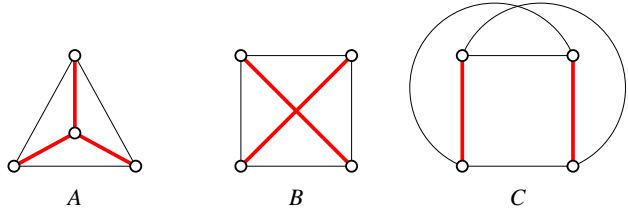}
\end{center}
\caption{The three different good
drawings of $K_{4}$, with $3$, $2$, and $2$ separations. The edge of each
separation is shown bold.}%
\label{types}
\end{figure}

We denote by $T_{A}$, $T_{B}$, and $T_{C}$ the number of induced subdrawings
of $D$ of type $A$, $B$, and $C$, respectively. Then
\begin{equation}
T_{A}+T_{B}+T_{C}=\binom{n}{4},\label{quads}%
\end{equation}
and since the subdrawings of types $B$ or $C$ are in one-to-one correspondence
with the crossings of~$D$, it follows that
\begin{equation}
\crg\left(  D\right)  =T_{B}+T_{C}.\label{crossings}%
\end{equation}
%
We count the number of separations in $D$ in two different ways: First, each
subdrawing of type $A$ has $3$ separations (the edge in each separation is
bold in Figure~\ref{types}), and each subdrawing of types $B$ or $C$ has 2
separations. This gives a total of $3T_{A}+2T_{B}+2T_{C}$ separations in $D$.
Second, each $k$-edge belongs to exactly $k (n-2-k)$ separations. Summing over
all $k$-edges for $0 \le k \le\lfloor n/2 \rfloor-1$ gives a total of
$\sum_{k=0}^{\lfloor n/2 \rfloor-1}k( n-2-k) E_{k}(D)$ separations in $D$.
Therefore
\begin{equation}
3T_{A}+2T_{B}+2T_{C}=\sum_{k=0}^{\lfloor n/2 \rfloor-1}k\left(  n-2-k\right)
E_{k}(D).\label{separations}%
\end{equation}
Finally, subtracting Equation~(\ref{separations}) from three times
Equation~(\ref{quads}) we get
\[
T_{B}+T_{C}=3\binom{n}{4}-\sum_{k=0}^{\lfloor n/2 \rfloor-1}k\left(
n-2-k\right)  E_{k}(D),
\]
and thus by Equation~(\ref{crossings}) we obtain the claimed result.
\end{proof}

For $0\leq k\leq\lfloor n/2\rfloor-1$ and $D$ a good drawing of $K_{n}$, we
define the set of ${\leq}k$\emph{-edges} of $D$ as all $j$-edges in $D$ for
$j=0,\ldots,k$. The number of ${\leq}k$-edges of $D$ is denoted by
\[
E_{{\leq}k}\left(  D\right)  :=\sum\limits_{j=0}^{k}E_{j}\left(  D\right)  .
\]
Similarly, we denote the number of ${\leq}{\leq}k$\emph{-edges }of $D$ by
\[
E_{{\leq}{\leq}k}\left(  D\right)  :=\sum\limits_{j=0}^{k}E_{\leq j}\left(
D\right)  =\sum\limits_{j=0}^{k}\sum\limits_{i=0}^{j}E_{i}\left(  D\right)
=\sum\limits_{i=0}^{k}\left(  k+1-i\right)  E_{i}\left(  D\right)  .
\]
To avoid special cases we define ${E_{\leq\leq-1}(D)=E_{\leq\leq-2}(D)=0.}$

The following result restates Theorem~\ref{CrvsEdges} in terms of the number
of ${\leq}{\leq}k$-edges.

\begin{proposition}
\label{IdentityForAtmostAtmost}Let $D$ be a good drawing of $K_{n}$. Then
\[
\crg(D)=2\sum\limits_{k=0}^{\lfloor n/2\rfloor-2}E_{{\leq}{\leq}k}(D)-\frac
{1}{2}\binom{n}{2}\left\lfloor \frac{n-2}{2}\right\rfloor -\frac{1}%
{2}(1+(-1)^{n})E_{\leq\leq\left\lfloor n/2\right\rfloor -2}(D).
\]

\end{proposition}

\begin{proof}
First note that for $2\leq k\leq\lfloor n/2\rfloor-1$ we have that $E_{{\leq
}{\leq}k}(D)-E_{{\leq}{\leq}k-1}(D)=E_{{\leq}k}(D)$ and $E_{{\leq}%
k}(D)-E_{\leq k-1}(D)=E_{k}(D)$. Thus
\[
E_{k}\left(  D\right)  =E_{{\leq}{\leq}k}\left(  D\right)  -2E_{\leq\leq
k-1}\left(  D\right)  +E_{{\leq}{\leq}k-2}\left(  D\right)  .
\]
We rewrite the last term in Theorem~\ref{CrvsEdges}.
\begin{align*}
&  \sum\limits_{k=0}^{\lfloor n/2\rfloor-1}k(n-2-k)E_{k}(D)\\
&  =\sum\limits_{k=2}^{\lfloor n/2\rfloor-1}k(n-2-k)[E_{{\leq}{\leq}%
k}(D)-2E_{{\leq}{\leq}k-1}(D)+E_{{\leq}{\leq}k-2}(D)]\\
&  =\sum\limits_{k=0}^{\lfloor n/2\rfloor-3}%
(k(n-2-k)-2(k+1)(n-3-k)+(k+2)(n-4-k))E_{{\leq}{\leq}k}(D)\\
&  \qquad+\left(  \left\lfloor \frac{n}{2}\right\rfloor -1\right)
\left( n-1-\left\lfloor \frac{n}{2}\right\rfloor \right)
E_{\leq\leq\lfloor n/2\rfloor-1}(D)+(-2\left(  \left\lfloor
\frac{n}{2}\right\rfloor -1\right)
\left(  n-1-\left\lfloor \frac{n}{2}\right\rfloor \right) \\
&  \qquad+\left(  \left\lfloor \frac{n}{2}\right\rfloor -2\right)
\left( n-\left\lfloor \frac{n}{2}\right\rfloor \right)
)E_{\leq\leq\lfloor
n/2\rfloor-2}(D)\\
&  =-2\sum\limits_{k=0}^{\lfloor n/2\rfloor-3}E_{{\leq}{\leq}k}(D)+\left(
\left\lfloor \frac{n}{2}\right\rfloor -1\right)  \left(  n-1-\left\lfloor
\frac{n}{2}\right\rfloor \right)  E_{\leq\leq\left\lfloor n/2\right\rfloor
-1}(D)\\
&  \qquad+(-2\left(  \left\lfloor \frac{n}{2}\right\rfloor -1\right)  \left(
n-1-\left\lfloor \frac{n}{2}\right\rfloor \right)  +\left(  \left\lfloor
\frac{n}{2}\right\rfloor -2\right)  \left(  n-\left\lfloor \frac{n}%
{2}\right\rfloor \right)  )E_{\leq\leq\left\lfloor n/2\right\rfloor -2}(D).
\end{align*}
Since $E_{\leq\leq\lfloor n/2\rfloor-1}(D)=E_{\leq\leq\lfloor
n/2\rfloor -2}(D)+E_{\leq\lfloor n/2\rfloor-1}(D)=E_{\leq\leq\lfloor
n/2\rfloor -2}(D)+\tbinom{n}{2}$, it follows by
Theorem~\ref{CrvsEdges} that
\begin{align*}
\crg\left(  D\right)   &  =3\binom{n}{4}-\sum\limits_{k=0}^{\left\lfloor
n/2\right\rfloor -1}k\left(  n-2-k\right)  E_{k}\left(  D\right)  =3\binom
{n}{4}+2\sum\limits_{k=0}^{\left\lfloor n/2\right\rfloor -3}E_{{\leq}{\leq}%
k}\left(  D\right) \\
&  +\left(  n+1-2\left\lfloor \frac{n}{2}\right\rfloor \right)  E_{\leq
\leq\left\lfloor n/2\right\rfloor -2}(D)-\left(  \left\lfloor \frac{n}%
{2}\right\rfloor -1\right)  \left(  n-1-\left\lfloor \frac{n}{2}\right\rfloor
\right)  \binom{n}{2}\\
&  =2\sum\limits_{k=0}^{\left\lfloor n/2\right\rfloor -3}E_{\leq\leq k}\left(
D\right)  -\frac{1}{2}\binom{n}{2}\left\lfloor \frac{n-2}{2}\right\rfloor +%
\begin{cases}
E_{\leq\leq\left\lfloor n/2\right\rfloor -2}(D) & \text{if $n$ is even},\\
2E_{\leq\leq\left\lfloor n/2\right\rfloor -2}(D) & \text{if $n$ is odd,}%
\end{cases}
\end{align*}
which is equivalent to the claimed result.
\end{proof}

\section{The 2-page crossing number}

\label{s:2page}

We are concerned with $2$-page book drawings of $K_n$. Obviously any
line can be chosen as the spine, and for the rest of the paper we
will assume that the spine is the $x$-axis. Moreover, by
topological equivalence, we will assume that the vertices are
precisely the points with coordinates $(1,0),(2,0),\ldots,(n,0)$.

Consider a 2-page book drawing $D$ of $K_n$, and label the vertices
$1,2,\ldots,n$ from left to right. Our interest lies in crossing
optimal drawings, and it is readily seen that in every such drawing,
none of the edges $(1,2), (2,3), \ldots, (n-1,n), (n,1)$ is crossed.
Thus we may choose to place each of these edges in either the upper
closed halfplane (page) or in the lower closed halfplane (page).
Moreover, we may choose to place each of the edges $(1,2), (2,3),
\ldots, (n-1,n)$ completely on the spine, and this is the convention
we shall adopt for the rest of the paper. The edge $(n,1)$ may be
placed indistinctly in the upper page or in the lower page, and for
the rest of the paper we adopt the convention that it is place in
the upper page. Moreover, because we are only concerned with
good drawings, we assume without loss of generality that the rest of
the edges are semicircles.

Color the edges above or on
the spine blue and below the spine red, respectively. We construct an upper
triangular matrix which corresponds to the coloring of these edges, see
Figure~\ref{diagram}. We call this the \emph{2-page matrix} of $D$ and denote
it by $M\left(  D\right)  $. Label the columns of the 2-page matrix with
$2,\ldots,n$ from left to right and the rows with $1,2,\ldots,n-1$ from top to
bottom. For $i<j$ an entry $(i,j)$ (row,column) in the 2-page matrix $M(D)$ is
a point with the same color as the edge $ij$ in the drawing~$D$.

\vglue 0.2 cm
\noindent{\bf Remark.} {\em It follows from the convention laid out above that for
  every 2-page book drawing $D$, the entries $(1,2), (2,3), \ldots,
  (n-1,n)$ and $(1,n)$ in $M(D)$ are all blue.}
\vglue 0.2 cm
\begin{figure}[ptbh]
\begin{center}
\includegraphics[scale=1, trim=0in 0.2in 0in 0.1in,
clip=true ] {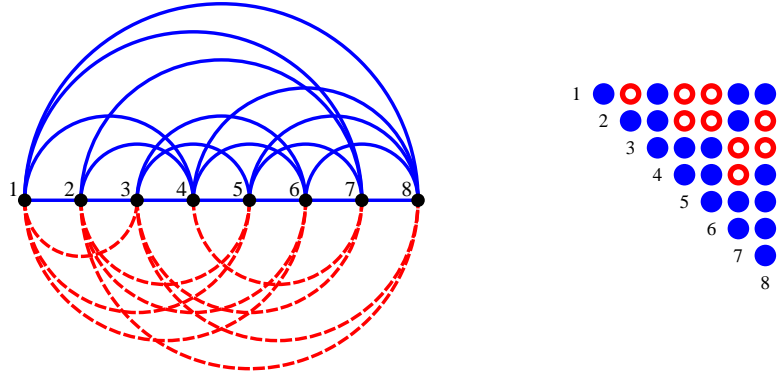}
\end{center}
\caption{Two-colored diagram for a 2-page book drawing $D$ of $K_{8}$ and the
corresponding 2-page matrix $M(D)$. Solid dots and lines represent blue edges.
Open dots and dashed lines represent red edges. From our
  convention to place the edges $(1,2),(2,3),\ldots,(n-1,n)$ on the
  spine and the edge $(1,n)$ in the upper page, it follows that all the entries in the main diagonal, as well
  as the upper right corner entry, are blue.
}%
\label{diagram}%
\end{figure}

We start by proving some basic properties of the 2-page matrix.

\begin{lemma}
\label{lem:kedge2matrix} Let $D$ be a 2-page book drawing of $K_{n}$. For
$1\leq i<j\leq n$, let $k$ be the sum of the number of points to the right
plus the number of points above the entry $(i,j)$ in the 2-page matrix of $D$,
which have the same color as $(i,j)$. Then the edge $ij$ is a $k$-edge. (It is
possible to have $k>\left\lfloor n/2\right\rfloor -1$.)
\end{lemma}

\begin{proof}
Let $1\leq i<j\leq n$ and assume that the edge $ij$ is blue (red). We count
the number of points $l$ in $D$ to the left (right) of $ij$. For
$l\not \in \{i,j\}$ the triangle $ijl$ is oriented counter-clockwise
(clockwise) if and only if either $l<i$ and the edge $lj$ is blue (red), or
$l>j$ and the edge $il$ is blue (red). In the first case these edges
correspond to blue (red) points above the entry $(i,j)$, and in the second
case to blue (red) points to the right of the entry $(i,j)$, respectively.
\end{proof}

In view of Lemma~\ref{lem:kedge2matrix} we say that the point in the
entry $(i,j)$ of the 2-page matrix of $D$ \emph{represents} a
$k$-edge if $ij$ is a $k $-edge (or an $(n-2-k)$-edge) in $D$.

\begin{lemma}
\label{manyofeach} For $k < n/2-1$ and for $1\leq j\leq k+1$, in the 2-page
matrix of a drawing $D$ of $K_{n}$ there are at least $2\left(  k+2-j\right)
$ points in row $j$ representing ${\leq}k$-edges. Similarly, for $n-k\leq
j\leq n$ there are at least $2\left(  k+1-n+j\right)  $ points in column $j$
representing ${\leq}k$-edges.
\end{lemma}

\begin{proof}
For $1\leq j\leq k+1$, in row $j$ the rightmost $k+2-j$ points of each color
represent ${\leq}k$-edges as they have at most $k+1-j$ points of their color
to the right and at most $j-1$ on top. So if each color appears at least
$k+2-j$ times in row $j$, we have guaranteed $2\left(  k+2-j\right)  $ ${\leq
}k$-edges in row $j$. If one of the colors appears fewer than $k+2-j $ times,
so that there are $k+2-j-e$ blue points in row $j$ for some $1\leq e\leq
k+2-j$, then there are $n-j-\left(  k+2-j-e\right)  =n-2-k+e$ red points in
this row. In this case we claim that also the leftmost $e$ red points in this
row represent ${\leq}k$-edges. In fact, for $1\leq i\leq e$, the $i$-{th} red
point (from the left) in row $j$, has exactly $n-2-k+e-i$ red points to the
right and perhaps more red points on top. Since for $n \geq2$ we have
$n-2-k+e-i \geq n/2-k$, this $i$-{th} red point also represents a ${\leq}%
k$-edge. The equivalent result holds for the rightmost $k+1$ columns.
\end{proof}

\begin{lemma}
\label{twoofeach} For $0 \leq j < n/2 -1$, in the 2-page matrix of a drawing
$D$ of $K_{n}$ there are two points in column $n$ which correspond to
$j$-edges in $D$. For $n$ even there exists one such point in column $n$
corresponding to an $(n/2 -1)$-edge in $D$.
\end{lemma}

\begin{proof}
We follow the lines of the proof of Lemma~\ref{manyofeach}. Consider the
points in column $n$ in order from top to bottom. By
Lemma~\ref{lem:kedge2matrix} the $i$-{th} vertex of a color corresponds to an
$(i-1)$-edge. Thus, if there are at least $j+1$ vertices for each color we are
done. Otherwise assume without loss of generality that there are $j+1-e$ blue
points in column $n$ for some $1 \leq e \leq j+1$. Then there are $n-1-(j+1-e)
= n-j+e-2$ red points in this column. For $1 \leq i \leq\lfloor n/2 \rfloor$
the $i$-{th} red point corresponds to an $(i-1)$-edge, and for $\lfloor n/2
\rfloor+1 \leq i \leq n-j+e-2$ the $i$-{th} red point corresponds to an $(i-1)
= (n-i-1)$-edge. Thus we get two red points corresponding to $j$-edges for
$i=j+1$ and $i=n-j-1$. Finally, observe that these two points are different
for $j < n/2 -1$. For $n$ even we get only one such point for $j=n/2-1$.
\end{proof}

The next theorem gives a lower bound on the number of ${\leq}{\leq}k$-edges,
which will play a central role in deriving our main result. We need the
following definitions. Let $D$ be a good drawing of $K_{n}$. Let $l$
be a vertex of $K_n$, and let $D^{\prime}$ be
the (evidently, also good) drawing of $K_{n-1}$ obtained by
deleting from $D$ the vertex $l$ and its adjacent edges.
Note that a $k$-edge $ij$ in $D^{\prime}$ is a
$k$-edge or a $(k+1)$-edge in $D$. Indeed, if $ij$ has exactly $k$ points to
its right in $D^{\prime}$ (an equivalent argument holds if the $k$ points are
on its left), then there are $k$ or $k+1$ points to the right of $ij$ in $D$
depending on whether $l$ is to the left or to the right, respectively, of
$ij$. We say that a $k$-edge in $D$ is $\left(  D,D^{\prime}\right)
$-\emph{invariant} if it is also a $k$-edge in $D^{\prime}$. Whenever it is
clear what $D$ and $D^{\prime}$ are, we simply say that an edge is invariant.
A $\left(  D,D^{\prime}\right)  $\emph{-invariant }$\leq k$\emph{-edge} is a
$\left(  D,D^{\prime}\right)  $-invariant $j$-edge for some $0\leq j\leq k\leq
n/2-1$. Denote by $E_{\leq k}(D,D^{\prime})$ the number of $(D,D^{\prime}%
)$-invariant $\leq k$-edges.

\begin{theorem}
\label{ThreeThrees}Let $n\geq3$. For every 2-page book drawing $D$ of $K_{n} $
and $0\leq k < n/2 -1$, we have
\[
E_{{\leq}{\leq}k}\left(  D\right)  \geq3\tbinom{k+3}{3}.
\]

\end{theorem}

\begin{proof}
We proceed by induction on $n$. The induction base $n=3$ holds trivially. For
$n\geq4$, consider a 2-page book drawing $D$ of $K_{n}$ with horizontal spine
and label the vertices from left to right with $1,2,\ldots,n$. Remove the
point $n$ and all incident edges to obtain a 2-page book drawing $D^{\prime}$
of $K_{n-1}$. To bound $E_{\leq{\leq}k}\left(  D\right)  $, recall that
\begin{equation}
E_{{\leq}{\leq}k}\left(  D\right)  =\sum\limits_{j=0}^{k}\left(  k+1-j\right)
E_{j}\left(  D\right)  .\label{AtmostAtmostIdent}%
\end{equation}
All edges incident to $n$ are in $D$ but are not in $D^{\prime}$. In fact, by
Lemma~\ref{twoofeach}, there are two $j$-edges adjacent to the vertex $n$ for
each $0\leq j\leq k\leq\left\lfloor n/2\right\rfloor -2$. These edges
contribute with $2\sum_{j=0}^{k}(k+1-j)=2\tbinom{k+2}{2}$ to
Equation~(\ref{AtmostAtmostIdent}). We next compare
Equation~(\ref{AtmostAtmostIdent}) to
\begin{equation}
E_{{\leq}{\leq}k-1}\left(  D^{\prime}\right)  =\sum\limits_{j=0}^{k-1}\left(
k-j\right)  E_{j}\left(  D^{\prime}\right)  .\label{Atmostlower}%
\end{equation}
Any edge contributing to Equation~(\ref{Atmostlower}) also contributes to
Equation~(\ref{AtmostAtmostIdent}), but possibly with a different value. As
observed before, a $j$-edge in $D^{\prime}$ is a $j$-edge or a $\left(
j+1\right)  $-edge in $D$. A $j$-edge in $D^{\prime}$ contributes to
Equation~(\ref{Atmostlower}) with $k-j$. A $j$-edge and a $(j+1)$-edge in $D$
contribute to Equation~(\ref{AtmostAtmostIdent}) with $k+1-j$ and $k-j$,
respectively. This is a gain of $+1$ or $0$, respectively, towards $E_{{\leq
}{\leq}k}(D)$ when compared to $E_{{\leq}{\leq}k-1}(D^{\prime})$. Finally, a
$k$-edge in both $D$ and $D^{\prime}$ does not contribute to
Equation~(\ref{Atmostlower}) and contributes to
Equation~(\ref{AtmostAtmostIdent}) with $+1.$ Therefore
\[
E_{{\leq}{\leq}k}(D)=E_{{\leq}{\leq}k-1}(D^{\prime})+2\binom{k+2}{2}+E_{\leq
k}(D,D^{\prime}).
\]
By induction hypothesis, $E_{{\leq}{\leq}k-1}(D^{\prime})\geq3\tbinom{k+2}{3}$
and thus
\[
E_{{\leq}{\leq}k}(D)\geq3\binom{k+2}{3}+2\binom{k+2}{2}+E_{\leq k}%
(D,D^{\prime})=3\binom{k+3}{3}-\binom{k+2}{2}+E_{\leq k}(D,D^{\prime}).
\]
We finally prove that
\begin{equation}
E_{\leq k}(D,D^{\prime})\geq\binom{k+2}{2}.\label{lastbinomial}%
\end{equation}
In fact, we prove that for each $1\leq j\leq k+1$ there are at least
$k+2-j$ points in row $j$ of $M(D)$ that represent
$(D,D^{\prime})$-invariant $\leq k$-edges. Suppose that the edge
$jn$ is blue (the equivalent argument holds when $jn$ is red). Then
any red point in row $j$ with $i\leq k$ red points above or to its
right in $M(D)$ represents a $(D,D^{\prime})$-invariant $i$-edge;
and any blue point in row $j$ with $i\geq n-2-k$ blue points above
or to its  right represents a $(D,D^{\prime})$-invariant
$(n-2-i)$-edge. Thus, the first $k+2-j$ red points from the right in
row $j$ (if they exist) represent $(D,D^{\prime})$-invariant $\leq
k$-edges as they have at most $k+2-j-1$ red points to the right and
at most $j-1$ red points above in both $M\left(  D\right)  $ and
$M\left(  D^{\prime}\right)  $. If there are fewer than $k+2-j$ red
points in row $j$ of $M(D)$, say $k+2-j-e$ for some $1\leq e\leq
k+2-j$, then the first $e$ blue points in row $j$ of $M(D)$ from the
left represent ${\leq}k$-edges, because they have at least
$n-j-e\geq n-j-k-2+j=n-k-2$ blue points to their right. Hence there
are at least $k+2-j-e$ red points and at least $e$ blue points (for
a total of at least $k+2-j$ points) that represent
$(D,D^{\prime})$-invariant $\leq k$-edges in row $j$ of $M(D)$.
Summing over all $1\leq j\leq k+1$, we get that
\[
E_{\leq k}(D,D^{\prime})\geq\sum\limits_{j=1}^{k+1}\left(
k+2-j\right) =\binom{k+2}{2}.\qedhere
\]

\end{proof}

We are now ready to prove our main result, namely that the 2-page crossing
number of $K_{n}$ is $Z(n)$.

\begin{theorem}
\label{TheResult} For every positive integer $n$, $\nu_{2}(K_{n})=Z(n)$.

\end{theorem}

\begin{proof}
The cases $n=1$ and $n=2$ are trivial. Let $n\geq3$. As we mentioned
above, 2-page book drawings with $Z\left(  n\right)  $ crossings
were constructed by Bla\v{z}ek and Koman\ \cite{BC64} (see also Guy
et al.\ \cite{Guy68}, Damiani et al.\ \cite{DaDASa}, Harborth
\cite{Harb02}, and Shahrokhi et al.\ \cite{SSSV96}.) These drawings
show that $\nu_{2}\left( K_{n}\right)  \leq Z\left( n\right)  $. For
the lower bound, let $D$ be a 2-page book drawing of $K_{n}$. Using
Proposition~\ref{IdentityForAtmostAtmost} and
Theorem~\ref{ThreeThrees}, we obtain
\begin{align*}
\crg\left(  D\right)   &  \geq2\sum\limits_{k=0}^{\left\lfloor
n/2\right\rfloor -2}3\binom{k+3}{3}-\frac{1}{2}\binom{n}{2}\left\lfloor
\frac{n-2}{2}\right\rfloor -\frac{3}{2}\left(  1+\left(  -1\right)
^{n}\right)  \binom{\left\lfloor \frac{n}{2}\right\rfloor +1}{3}\\
&  =6\binom{\left\lfloor \frac{n}{2}\right\rfloor +2}{4}-\frac{1}{2}\binom
{n}{2}\left\lfloor \frac{n-2}{2}\right\rfloor -\frac{3}{2}\left(  1+\left(
-1\right)  ^{n}\right)  \binom{\left\lfloor \frac{n}{2}\right\rfloor +1}{3}\\
&  =%
\begin{cases}
\tfrac{1}{64}\left(  n-1\right)  ^{2}\left(  n-3\right)  ^{2} & \text{if $n$
is odd},\\
\tfrac{1}{64}n\left(  n-2\right)  ^{2}\left(  n-4\right)  & \text{if $n$ is
even,}%
\end{cases}
=Z(n).
\end{align*}\vspace{-32.5pt}
\[
 \qedhere
\]
\end{proof}

\section{Crossing optimal configurations}\label{s:OptConfig}

In all this section $D$ denotes a 2-page book drawing of $K_{n}$ and $M(D)$
its 2-page matrix. We say that $D$ is \emph{crossing optimal}  if $\nu_{2}(D)=Z(n)$. Theorem \ref{mainstructure} in
Subsection~\ref{section:mainstructure} describes
the general structure of the crossing optimal 2-page book drawings of $K_{n}$.
We use it to prove that, up to the equivalence described below, there is a
unique crossing optimal 2-page book drawing of $K_{n}$ when $n$ is even and,
in contrast, there exists an exponential number of non-equivalent crossing optimal 2-page book drawings of $K_{n}$ when $n$ is odd.

\subsection{Equivalent drawings}
\label{section:equivalent}

Let $D$ be a 2-page book drawing of $K_{n}$. Recall that we
are assuming that the vertices of $D$ are the points $\left\{
\left( i,0\right) :1\leq i\leq n\right\}$. Consider the following
transformation $f$ that results in the 2-page book drawing $f\left(
D\right) $ of $K_{n}$: move the vertex $(1,0)$ to the point $(n,0)$,
and for every $2\leq k\leq n$ move the vertex $(k,0)$ to the vertex
$(k-1,0)$. That is, if an edge $1j$ was drawn above (below) the
spine in $D$, then the edge $(j-1)(n)$ is drawn above (below) the
spine in $f\left( D\right) $; for all other edges $ij$ with $1<i<j
\leq n$, if $ij$ was drawn above (below) the spine in~$D$, then the
edge $(i-1)(j-1)$ is drawn above (below) the spine in $f\left(
D\right) $. Note that $D$ and $f\left(
  D\right) $ have the same number of crossings, and
$f^{n}(D)=D$. There are two other natural transformations of a
drawing $D$: A vertical reflection $g(D)$ about the line with
equation $x=n/2$ and a horizontal reflection $h(D)$ about the spine
(or $x$-axis). In $g(D)$ an edge $ij$ is drawn above (below) the
spine if the edge $(n+1-j)(n+1-i)$ is drawn above (below) the spine
in $D $. In $h(D)$ an edge $ij$ is drawn above (below) the spine if
the edge $ij$ is drawn below (above) the spine in~$D$. Note that
$g^{2}(D)=h^{2}(D)=D$. Given a 2-page drawing $D$, all drawings
obtained by compositions of these transformations from $D$ are said
to be \emph{equivalent }to $D$. All drawings obtained this way are
topologically isomorphic (homeomorphic) and thus they all have the
same number of crossings as $D$. The group spanned by these
transformations is isomorphic to the direct sum of the dihedral
group $D_{2n}$ and the group with 2 elements $\mathbb{Z}_{2}$. The
set $\{f,g,h\}$ is a set of generators such that
$g^{2}=h^{2}=f^{n}=1$, $g\circ f=f^{-1}\circ g$, $h\circ f=f\circ
h$, and $g\circ h=h\circ g$. Thus the $4n$ transformations in the
group can be parametrized by $h^{a}\circ g^{b}\circ f^{i}$ with
$i\in\{0,1,\ldots,n-1\}$ and $a,b\in\{0,1\}$.

Now we describe these transformations in the 2-page matrix diagram of $D$: To
obtain $M\left(  f\left(  D\right)  \right)  $ from $M\left(  D\right)  $, we
simply rotate $90$ degrees counterclockwise the first row of $M\left(  D\right)  $ and use it as the $n^{th}$
column of $M\left(  f\left(  D\right)  \right)  $. The diagram $M\left(
g\left(  D\right)  \right)  $ is obtained from $M\left(  D\right)  $ by
reflecting with respect to the diagonal $\{(i,n+1-i):1\leq i\leq\lfloor
n/2\rfloor\}$. Finally, $M(h(D))$ is obtained by switching the color of every
point except those that join consecutive vertices on the spine or the point
$(1,n)$. We can place $M\left(  D\right)  $ and $M\left(  f\left(  D\right)
\right)  $ together so that the part they have in common overlaps. Doing this
for $M\left(  f^{m}\left(  D\right)  \right)  $ for all integers $m$ we obtain
a periodic double infinite strip with period $n$ and with a horizontal section
that is $n-1$ units wide. We call this the \emph{strip diagram }of $D$, or of
$f^{m}\left(  D\right)  $ for any integer $m$. (See Figure~\ref{strip}.) Any right triangular region
with the same dimensions as $M\left(  D\right)  $ obtained from the strip
diagram of $D$ by a horizontal and a vertical cut is the matrix diagram of a
drawing equivalent to $D$ and thus it has the same number of crossings as $D$.

\begin{figure}
[ptb]
\begin{center}
\includegraphics[
scale=1]
{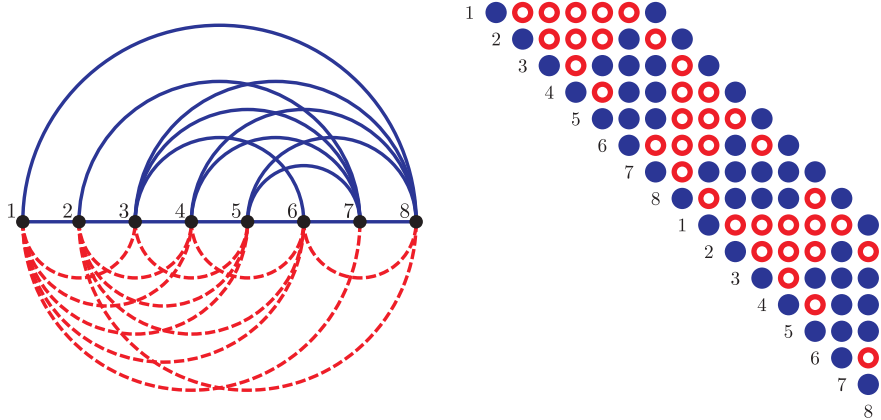}
\caption{A $2$-page drawing of $K_8$ and its strip diagram.}
\label{strip}
\end{center}
\end{figure}

\subsection{Properties of crossing optimal drawings}

We start with a couple of definitions. Consider the entry $\left(
i,j\right) $ of $M(D)$. We order the entries in row $i$ to the left
of $\left( i,j\right)  $ as follows: first all entries, from right
to left, whose color differs to that of $(i,j)$, followed by all
other entries (those with the same color as $(i,j)$) from left to
right. This is called the \emph{order associated to }$\left(
i,j\right)  $. Observe that this is the order in which the edges
$il$ ($i<l<j$) appear in the 2-page drawing, ordered bottom to top
if the edge $ij$ is blue and top to bottom if the edge $ij$ is red.
Let $c$ be an integer such that $0\leq c\leq n-1$. Denote by $D_{c}$
the subgraph of $D$ obtained by deleting the $c $ right-most points
of $D$, or equivalently, $M(D_{c})$ is obtained by deleting the last
$c$ columns of $M\left(  D\right)  $. The following results strongly
rely on the proof of Theorem \ref{ThreeThrees}.

\begin{lemma}
\label{OrderObservation}Suppose that $l\geq i+m+1$ for some integers
$1\leq i<l<j\leq n$ and $1\leq m<j-i$. The entry $(i,l)$ is one of
the first $m$ entries in the order associated to $(i,j)$ if and only
if $(i,l)$ and $(i,j)$ have different colors.
\end{lemma}

\begin{proof}
Note that if $(i,l)$ and $(i,j)$ have the same color, then all
entries to the left of $(i,l)$ come before $(i,l)$ in the order
associated to $(i,j)$.
\end{proof}

\begin{lemma}
\label{allprevious}Let $p$ be an integer such that $0\leq p\leq\left\lfloor
n/2\right\rfloor -2$. Suppose that $E_{\leq k}(D,D_{1})=\tbinom{k+2}{2}$ for
all $0\leq k\leq p$. Then $M\left(  D\right)  $ satisfies that for $1\leq
i\leq p+1$ in row $i$ there is exactly one $(D,D_{1})$-invariant $k$-edge for
each $i-1\leq k\leq p$, and there are no $(D,D_{1})$-invariant $\left(  \leq i-2\right)
$-edges. In all other rows there are no $(D,D_{1})$-invariant $\leq p$-edges.
\end{lemma}

\begin{proof}
In what follows all invariant edges are $\left(  D,D_{1}\right)  $-invariant
edges. For $k=0$ the statement implies that there is a unique invariant
0-edge and it appears in row 1. Note that this edge corresponds to the first entry in the order
associated to $(1,n)$ in $M(D)$. Following the proof of Theorem
\ref{ThreeThrees}, $E_{\leq k}(D,D_{1})=\tbinom{k+2}{2}$ implies that for all
$1\leq i\leq k+1$ there are exactly $k+2-i$ invariant $\leq k$-edges in row $i
$ of $M\left(  D\right)  $, and for $k+2\leq i\leq n-1$ there are no invariant
$\leq k$-edges in row $i$ of $M(D)$. The second part implies that there are no
invariant $k$-edges in row $i$ for all $k+2\leq i\leq n-1$ and $0\leq k\leq p$. Similarly,
$E_{\leq k-1}(D,D_{1})=\tbinom{k+1}{2}$ implies that for all $1\leq i\leq k$
there are exactly $k+1-i$ invariant $\left(  \leq k-1\right)  $-edges in row
$i $ of $M\left(  D\right)  $. Therefore for all $1\leq i\leq k$ there is
exactly $\left(  k+2-i\right)  -\left(  k+1-i\right)  =1$ invariant $k$-edge,
and for $i=k+1$ there is exactly $k+2-\left(  k+1\right)  =1$ invariant $\leq
k$-edge and no invariant $\left(  \leq k-1\right)  $-edge in row $i$ of
$M\left(  D\right)  $. Therefore, there is exactly one invariant $k$-edge in
row $k+1$.
\end{proof}

\begin{lemma}
\label{orderoriginal}Let $p$ be an integer such that $0\leq p\leq\left\lfloor
n/2\right\rfloor -2$.

i) Suppose that for some $1\leq i\leq p+1$ row $i$ of $M(D)$ has
exactly one $(D,D_{1})$-invariant $k$-edge for each $i-1\leq k\leq p$ and no
$(D,D_{1})$-invariant $\leq(i-2)$-edges. If the entry $(i,n)$ in $M(D)$ is
blue (red), then the $m^{th}$ entry in row $i$ in the order associated to
$(i,n)$ has at least $\min\{p+2-m,i-1\}$ red (blue) entries above for every
$1\leq m\leq\min\{p+1,n-i-1\}$.

ii) Suppose that for some $i\geq p+2$ row $i$ of $M(D)$ does not have
$(D,D_{1})$-invariant $\leq p$-edges. If the entry $(i,n)$ in $M(D)$ is blue
(red), then the $m^{th}$ entry in row $i$ in the order associated to $(i,n)$
has at least $p+2-m$ red (blue) entries above for every $1\leq m\leq
\min\{p+1,n-i-1\}$.
\end{lemma}

\begin{proof}
In what follows invariant edges refer to $(D,D_{1})$-invariant edges.
Denote by $(i,e_{m})$ the
$m^{th}$ entry in the order associated to $(i,n)$. Note that if $(i,e_{m})$
and $(i,n)$ have opposite colors and the number of points above plus the
number of points to the right of $(i,e_{m})$ with the same color as
$(i,e_{m})$ is at most $p$, then $(i,e_{m})$ is an invariant $\leq p$-edge.
Similarly, if $(i,e_{m})$ and $(i,n)$ have the same color and the number of
points above plus the number of points to the right of $(i,e_{m})$ with the
same color as $(i,e_{m})$ is more than $n-2-p$, then $(i,e_{m})$ is an
invariant $\leq p$-edge.

Suppose that the entry $(i,n)$ of $M(D)$ is blue (red).

(i) If $(  i,e_{1})  $ is red (blue), then it does not have red
entries to its right and it has at most $i-1$ red (blue) entries
above. Since $i-1\leq p$, then $(  i,e_{1})  $ is an invariant $(
\leq i-1)  $-edge. Because there are no invariant $( \leq i-2)
$-edges in row $i$, it follows that $( i,e_{1})  $ is the unique
invariant $( i-1)  $-edge in row $i$ and thus all $i-1$ entries
above it are red (blue). Similarly, if the $( i,e_{1}) $ is blue
(red), then all entries in row $i$ are blue (red) and $( i,e_{1}) =(
i,i+1) $. Hence $( i,e_{1})  $ has $n-i-1$ blue (red) entries to its
right and perhaps some other blue (red) entries above. Since
$n-i-1\geq n-( p+1)  -1\geq n-2-p$, then $( i,e_{1})  $ is an
invariant $(  \leq i-1) $-edge. Because there are no invariant $(
\leq i-2) $-edges in row $i$, it follows that $( i,e_{1})  $ is the
unique invariant $( i-1)  $-edge in row $i$ and thus all $i-1$
entries above it are red (blue).

For $2\leq m\leq p+2-i$ assume that the entry $(  i,e_{m^{\prime}%
})  $ is an invariant $(  i-2+m^{\prime})  $-edge for every $1\leq
m^{\prime}\leq m-1$. Note that $i-1\leq i-2+m^{\prime}\leq p-1$.

If $(  i,e_{m})  $ is red (blue), then $(  i,e_{m^{\prime}%
})  $ is red (blue) for every $1\leq m^{\prime}\leq m-1$. So $(
i,e_{m})  $ has exactly $m-1$ red (blue) entries to its right and at
most $i-1$ red (blue) entries above, that is, $( i,e_{m})  $ is an
invariant $\leq(  i-2+m)  $-edge. By hypothesis there is a unique
invariant $k$-edge for every $i-1\leq k\leq p$ and among the first
$m-1$ entries there is exactly one invariant $k$-edge for each
$i-1\leq k\leq i-2+( m-1) =i-3+m$. So $(  i,e_{m}) $ is the unique
invariant $( i-2+m) $-edge (note that $1\leq i-2+m\leq p$) and thus
all the entries above it are red (blue).

If $(  i,e_{m})  $ is blue (red), then there are exactly $n-i+m$
blue (red) entries to its right and perhaps some others above it.
Since $n-i+m\geq n-i-(  p+2-i)  =n-p+2$, then $( i,e_{m})  $ is an
invariant $\leq(  i-2+m) $-edge. As before $( i,e_{m}) $ must be an
invariant $( i-2+m) $-edge and thus it must have only red (blue)
entries above.

We have already determined the unique invariant $k$-edge for each
$1\leq k\leq p$. So there are no more invariant $\leq p$-edges in
row $i$. For $p+3-i\leq m\leq\min\{  p+1,n-i-1\}  $, we prove that
the entry $( i,e_{m})  $ has at least $p+2-m=\min\{ p+2-m,i-1\}  $
red (blue) entries above.

If $(  i,e_{m})  $ is red (blue), then it has $m-1$ red (blue)
entries to its right. If $(  i,e_{m})  $ had less than $p+2-m$ (note
that $p+2-m\leq i-1$) red (blue) entries above, then it would be an
invariant $\leq p$-edge (because $(  m-1) +( p+1-m) =p$) getting a
contradiction.

If $(  i,e_{m})  $ is blue (red), then it has $n-i-m$ blue (red)
entries to its right. If $(  i,e_{m})  $ had less than $p+2-m$ red
(blue) entries above, then it would have a total of at least
$n-i-m+( i-1)  -(  p+1-m)  =n-2-p$ blue (red) entries above or to
its right, and thus it would be an invariant $\leq p$-edge getting a
contradiction.

(ii) The proof is the same as for the case $p+3-i\leq m\leq\min\{
p+1,n-i-1\}  $ in (i) as we only used that the $m^{th}$ entry in that
range was not an invariant $\leq p$-edge.
\end{proof}

\begin{lemma}
\label{frominduction}If $D$ is crossing optimal, then for $0\leq
j\leq\lfloor n/2\rfloor -2$ we have
\[
E_{{\leq}{\leq}k}(D_{j})=3\tbinom{k+3}{3}\text{ for all }0\leq k\leq
\lfloor n/2\rfloor -2-j.
\]

\end{lemma}

\begin{proof}
Since $D$ is crossing optimal, then equality must be achieved in the
proof of Theorem \ref{TheResult}, that is, $E_{{\leq}{\leq}k}( D)
=3\tbinom {k+3}{3}$ for all $0\leq k\leq\lfloor n/2\rfloor -2$. This
implies that equality must be achieved throughout the proof of
Theorem \ref{ThreeThrees}, in particular, $E_{{\leq}{\leq}k}( D_{1})
=3\tbinom{k+2}{3}$ for all $0\leq k\leq\lfloor n/2\rfloor -2,$ which
is equivalent to $E_{{\leq}{\leq}k}( D_{1})  =3\tbinom {k+3}{3}$ for
all $0\leq k\leq\lfloor n/2\rfloor -3$.

In general, for $0\leq j\leq\lfloor n/2\rfloor -2$, following the
proof of Theorem \ref{ThreeThrees}, $E_{{\leq}{\leq}k}( D_{j})
=3\tbinom{k+3}{3}$ for $1\leq k\leq\lfloor n/2\rfloor -2-j$ implies
that $E_{{\leq}{\leq}k-1}( D_{j+1}) =3\tbinom{k+2}{3}$ for $1\leq
k\leq\lfloor n/2\rfloor -2-j$, which is equivalent to $E_{{\leq
}{\leq}k}( D_{j+1}) =3\tbinom{k+3}{3}$ for $1\leq k\leq \lfloor
n/2\rfloor -2-j-1$.
\end{proof}

\begin{lemma}
\label{order}If $D$ is crossing optimal, then in $M(D)$ the $m^{th}$
entry in the order associated to $(  i,j)  $ has at least $\min\{
j-\lceil n/2\rceil -m,i-1\}  $ entries above with different color
than $(i,j)$ for all $1\leq m\leq\min\{ j-\lfloor n/2\rfloor
-1,j-i-1\}  .$
\end{lemma}

\begin{proof}
Consider the entry $(  i,j)  $ of $M(  D)  $. Because $D$ is
crossing optimal, it follows from Lemma \ref{frominduction} that
\[
E_{{\leq}{\leq}k}(  D_{n-j})  =3\tbinom{k+3}{3}\text{ for all }0\leq
k\leq\lfloor n/2\rfloor -2-(  n-j) =j-2-\lceil n/2\rceil .
\]

Consider row $i$ of $D_{n-j}$. (Note that $D_{n-j}$ has $j-1$ rows.)
If $1\leq i\leq j-1-\lceil n/2\rceil $, then by Lemma
\ref{allprevious} for $p=j-2-\lceil n/2\rceil $, the 2-page matrix
$M( D_{n-j})  $ satisfies that in row $i$ there is exactly one $(
D_{n-j},D_{n-j+1})  $-invariant $k$-edge for each $i-1\leq k\leq
j-2-\lceil n/2\rceil $ and there are no $(  D_{n-j}%
,D_{n-j+1})  $-invariant $(  \leq j-2-\lceil n/2\rceil ) $-edges.
Then by Lemma \ref{orderoriginal}(i) if the entry $( i,j)  $ in $M(
D)  $ (actually in $M( D_{n-j}) $ but we look at it as a submatrix
of $M( D)  $) is blue (red), then the $m^{th}$ entry in the order
associated to $( i,j)  $ has at least $\min\{ j-\lceil n/2\rceil
-m,i-1\} $ red (blue) entries above.

If $j-\lceil n/2\rceil \leq i\leq j-1$, then by Lemma
\ref{allprevious} for $p=j-2-\lceil n/2\rceil $, the 2-page matrix
$M(  D_{n-j})  $ satisfies that in row $i$ there are no $(
D_{n-j},D_{n-j+1})  $-invariant $( \leq j-2-\lceil n/2\rceil )
$-edges. Then by Lemma \ref{orderoriginal}(ii) if the entry $( i,j)
$ in $M( D)  $ is blue (red), then the $m^{th}$ entry in the order
associated to $(  i,j)  $ has at least $j-\lceil n/2\rceil
-m=\min\{j-\lceil n/2\rceil -m,i-1\}$ red (blue) entries above.
\end{proof}

\begin{corollary}
\label{OneAbove}If $D$ is crossing optimal, then for $2\leq
i\leq\lceil n/2\rceil $ and $\lceil n/2\rceil +2\leq j\leq n$, each
of the first $j-\lceil n/2\rceil -1$ entries in the order associated
to $(i,j)$ has at least one entry above with different color than
$(i,j)$.
\end{corollary}

\begin{proof}
Let $1\leq m\leq j-\lceil n/2\rceil -1$. Since $\lfloor n/2\rfloor $
and $i$ are at most $\lceil n/2\rceil $, then $m\leq\min\{ j-\lfloor
n/2\rfloor -1,j-i-1\}  $. Also $m\leq j-\lceil n/2\rceil -1$ and
$i\geq2$ imply that $\max\{j-\lceil n/2\rceil -m,i-1\}\geq1$. Thus
by Lemma \ref{order}, the $m^{th}$ entry in row $i$ in the order
associated to $(  i,j) $ has at least one entry above with different
color than $(i,j)$.
\end{proof}

\begin{corollary}
\label{AllAbove}If $D$ is crossing optimal, then for $n\geq3$,
$2\leq i\leq\lfloor n/2\rfloor -1,$ and $\lceil n/2\rceil +i\leq
j\leq n$, all entries above the first $j-i+1-\lceil n/2\rceil $
entries in the order associated to $(i,j)$ have different color than
$(i,j)$.
\end{corollary}

\begin{proof}
Let $1\leq m\leq j-i+1-\lceil n/2\rceil $. Since $i\geq2$ and
$n\geq3$, then $m\leq\min\{j-\lfloor n/2\rfloor -1,j-i-1\}$. Also
$m\leq j-i+1-\lceil n/2\rceil $ implies that $\max\{j-\lceil
n/2\rceil -m,i-1\}\geq i-1$. Thus by Lemma \ref{order}, the $m^{th}$
entry in row $i$ in the order associated to $(  i,j)  $ has at least
$i-1$ entries above, (i.e., all entries above it) with different
color than $(i,j)$ in $M( D) $.
\end{proof}

\begin{lemma}
\label{firstrowscolumns}Suppose that $D$ is crossing optimal and
$0\leq k\leq \lfloor n/2\rfloor -2$. Then all $\leq k$-edges of $D$
belong to the union of the first $k+1$ rows and the last $k+1$
columns of $M( D)  $.
\end{lemma}

\begin{proof}
Suppose by contradiction that the entry $(i,j)$ of $M(D)$ represents
a $k$-edge and is not in the first $k+1$ rows ($i\geq k+2$) or in
the last $k+1$ columns ($j\leq n-k-1$). Since $D$ is crossing
optimal, equality must be achieved in
Inequality~(\ref{lastbinomial}) and thus we have that all $(
D,D_{1}) $-invariant $\leq k$-edges belong to the first $k+1$
columns. So $(i,j)$ is not $( D,D_{1}) $-invariant, that is, $ij$ is
a $( k-1) $-edge in $D_{1}$. Equality in Theorem \ref{ThreeThrees}
implies that $E_{\leq\leq k-1}( D_{1}) =3\tbinom{k+2}{2}$ and as
before all $( D_{1},D_{2}) $-invariant $( \leq k-1)  $-edges belong
to the last $k$ columns of $M( D_{1}) $, that is, columns
$n-k,n-k+1,\ldots,n-1$ of $M( D)  $. So $ij$ is not a $(
D_{1},D_{2}) $-invariant edge, that is, $ij$ represents a $( k-2)
$-edge in $D_{2}$. In general, assuming that $ij$ is a $( k-l)
$-edge in $D_{l}$ and since $E_{\leq\leq k-l}( D_{l})
=3\tbinom{k-l+3}{2}$, then all $( D_{l},D_{l+1}) $-invariant $( \leq
k-l) $-edges belong to the last $k+1-l$ columns of $M( D_{l}) $
(i.e., columns $n-k,n-k+1,\ldots,n-l$ of $M(D)$). So $(i,j)$ is not
a $( D_{l},D_{l+1}) $-invariant edge, that is, $(i,j)$ represents a
$( k-l-1)  $-edge in $D_{l+1}$. When $l=k-1$, $(i,j)$ is a $0$-edge
in $M(  D_{k})  $ that is not in the last column of $M( D_{k}) $
(column $n-k$ of $M( D)  $). Since there are at least three
$0$-edges in the first column and row of $M( D_{k}) $ and $i\geq2$,
then $E_{\leq\leq0}( D_{k}) \geq4$, but $E_{\leq\leq0}( D_{k}) $
must be $3$, getting a contradiction.
\end{proof}

We extend the standard terminology from the geometrical setting, and
call a $(\lfloor n/2\rfloor -1)$-edge a \mbox{\em halving edge}.

\begin{lemma}
\label{halving}If $D$ is crossing optimal, then the entries
$(\lfloor n/2\rfloor ,\lceil n/2\rceil +1),(\lfloor n/2\rfloor
,\lfloor n/2\rfloor +1),$ and $(\lceil n/2\rceil ,\lceil n/2\rceil
+1)$ of $M(D)$ are halving edges.
\end{lemma}

\begin{proof}
This follows from Lemma \ref{firstrowscolumns} as all $\leq( \lfloor
n/2\rfloor -2)  $-edges of $D$ belong to the union of the first
$\lfloor n/2\rfloor -1$ rows (top to bottom) and the last $\lfloor
n/2\rfloor -1$ columns (left to right) of $D$. The entries $(\lfloor
n/2\rfloor ,\lceil n/2\rceil +1),(\lfloor n/2\rfloor ,\lfloor
n/2\rfloor +1),$ and $(\lceil n/2\rceil ,\lceil n/2\rceil +1)$ are
not in the first $\lfloor n/2\rfloor -1\ $rows or in the last
$\lfloor n/2\rfloor -1\ $columns.
\end{proof}

Lemma \ref{halving} guarantees that the entry $(i,i+1)$ in general,
and the entry $(i,i+2)$ when $n$ is odd, are halving lines in some
drawing equivalent to $D$. The next result states what this means in
$D$. We state it only for $1\leq i\leq\lfloor n/2\rfloor $ (but it
can be stated for $\lceil n/2\rceil \leq i\leq n$ as well) as it is
the only case we explicitly use later in the paper.

\begin{lemma}
\label{halvingMobius}Let $1\leq i\leq\lfloor n/2\rfloor $. If $D$ is
crossing optimal, then $M(  D)  $ satisfies that the number of blue
entries in%
\begin{align}
\{(r,i+1) &  :1\leq r\leq i-1\}\cup\{(i,c):i+2\leq c\leq i+\lceil
n/2\rceil \}\label{halvingMup}\\
\cup\{(i+1,c) &  :i+\lceil n/2\rceil +1\leq c\leq n\}\nonumber
\end{align}
is either $\lfloor n/2\rfloor -1$ or $\lceil n/2\rceil
-1$. If $n$ is odd, then the number of entries in%
\begin{align}
\{(r,i+2) &  :1\leq r\leq i-1\}\cup\{(i,c):i+3\leq c\leq i+\lceil
n/2\rceil \}\label{halvingUpseconddiag}\\
\cup\{(i+2,c) &  :i+\lceil n/2\rceil +1\leq c\leq n\}\nonumber
\end{align}
with the same color as the entry $(i,i+2)$ is either $\lfloor
n/2\rfloor -1$ or $\lfloor n/2\rfloor $.
\end{lemma}

\begin{proof}
In the strip diagram of $D$, the entry $(  i,i+1)  $ of $M(D) $
corresponds to the entry $(  \lfloor n/2\rfloor ,\lfloor n/2\rfloor
+1)  $ of $M( f^{i-\lfloor n/2\rfloor }(  D) )  $, see Figure
\ref{HalvingEquivalent} (left). Applying Lemma \ref{halving} to
$M(f^{i-\lfloor n/2\rfloor }( D) ) $ and noticing that the entries
of $M(D)$ in (\ref{halvingMup}) correspond to the entries above plus
the entries below the entry $(\lfloor n/2\rfloor ,\lfloor n/2\rfloor
+1)$ of $M(f^{i-\lfloor n/2\rfloor }(D))$ gives the result. The
proof of the second part is similar, see Figure
\ref{HalvingEquivalent} (right).
\end{proof}

\begin{figure}
[h]
\begin{center}
\includegraphics[
height=1.9424in,
width=5.6688in]
{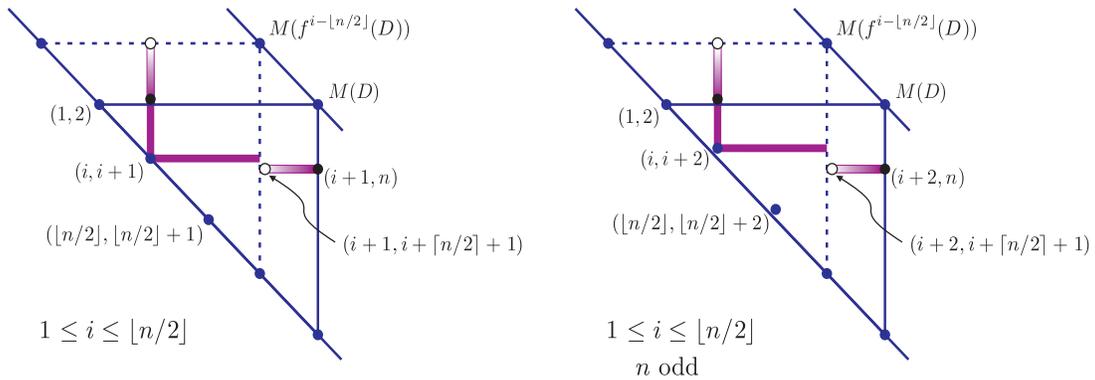}
\caption{A halving line in a drawing equivalent to $D$ seen in the matrix
$M(D)$. }
\label{HalvingEquivalent}
\end{center}
\end{figure}

\begin{lemma}
\label{start}If $D$ is crossing optimal, then there exists a drawing
$D^{\prime}$ equivalent to $D$ such that in $M(  D^{\prime}) $ the
$\lceil n/2\rceil $ entries $(  1,n) ,( 2,n) ,\ldots,$ and $( \lceil
n/2\rceil ,n)  $ are blue and the $\lfloor n/2\rfloor -1$ entries $(
1,\lceil n/2\rceil +1) $, $( 1,\lceil n/2\rceil +2) ,\ldots$, $(
1,n-1) $ are red.
\end{lemma}

\begin{proof}
For each integer $m$, let $e_{m}$ be the largest integer such that
the last $e_{m}$ entries in row $\lfloor n/2\rfloor $ of $M( f^{m}(
D)  )  $ have the same color. (These entries are $( \lfloor
n/2\rfloor ,n-e_{m}+1) $, $\ldots$, $( \lfloor n/2\rfloor ,n) $.)
Similarly, let $e_{m}^{\prime}$ be the largest integer such that the
first $e_{m}^{\prime}$ entries in column $\lceil n/2\rceil +1$ of
$M(  f^{m}( D)  ) $ have the same color. (These entries are $(
1,\lceil n/2\rceil +1) $, $\ldots$, $( e_{m}^{\prime},\lceil
n/2\rceil +1)  $.) Let $E=\max\{ e_{m},e_{m}^{\prime}:m\in\mathbb{Z}
\}  $. We claim that $E=\lceil n/2\rceil $. Indeed, suppose that
$E\leq\lceil n/2\rceil -1$ and without loss of generality assume
that $E=e_{m_{0}}$ for some integer $m_{0}$. (If
$E=e_{m_{0}}^{\prime} $, start with $g(  D)  $ instead of $D$.) Then
entry $( \lfloor n/2\rfloor ,n-e_{m_{0}})  $ has a different color
than the entries to its right, namely, $(  \lfloor n/2\rfloor
,n-e_{m_{0}}+1) ,\ldots,(  \lfloor n/2\rfloor ,n) $. By Lemma
\ref{order} (for $i=\lfloor n/2\rfloor $ and $j=n$) the entry $(
\lfloor n/2\rfloor ,n-e_{m_{0}})  $ has at least $\min\{ n-\lceil
n/2\rceil -1,\lfloor n/2\rfloor -1\}  =\lfloor n/2\rfloor -1$
entries above with the same color as $(  \lfloor n/2\rfloor
,n-e_{m_{0} })  $. But this means that $e_{m_{0}-1+\lfloor
n/2\rfloor -e_{m_{0}}}^{\prime}\geq e_{m_{0}}+1=E+1$, a
contradiction.

Because $E=e_{m_{0}}=\lceil n/2\rceil $, all entries in row $\lfloor
n/2\rfloor $ of $M(  f^{m_{0}}( D) ) $ are blue. By Lemma
\ref{halving} all entries in column $\lfloor n/2\rfloor +1$ of $M(
f^{m_{0}}( D) )  $ above the entry $( \lfloor n/2\rfloor ,\lfloor
n/2\rfloor +1)  $ are red. This implies that
$D^{\prime}=f^{m_{0}+\lfloor n/2\rfloor }( D) $ satisfies the
statement.
\end{proof}

\subsection{The structure of crossing optimal drawings}
\label{section:mainstructure}

We are finally ready to investigate the structure of crossing optimal
drawings. The next result is the workhorse behind
Theorems~\ref{OptimalEven} and ~\ref{OptimalOdd}, the main results in
this section. To help comprehension, we refer the reader to Figure~\ref{evenandodd}.

\begin{theorem}
\label{mainstructure}Let $n\geq6$,
$e=0$ for $n$ even and $e=1$ for $n$ odd,
and let $D$ be a crossing optimal 2-page book
drawing of $K_{n}$. Then there exists a drawing $D^{\prime}$ equivalent to $D$
such that $M(D^{\prime})$ satisfies:

\begin{enumerate}
\item\label{it:uno} for $4+e\leq s\leq\lfloor n/2\rfloor +1$ and $n+2+e\leq s\leq
n+\lfloor n/2\rfloor+1 $ the entry $(  r,s-r)  $ is blue for all
$\max\{1,s-n\}\leq r\leq (s-5)/2$;


\item\label{it:dos} for $\lceil n/2\rceil +2+e\leq s\leq n$ and $n+\lceil
n/2\rceil +2+e\leq s\leq\,2n-2-e$ the entry $(  r,s-r) $ is red for
all $\max\{1,s-n\}\leq r\leq (s-5)/2$ (except for $(1,n)$, which by
convention is blue);

\item\label{it:tres} for $n$ odd, the entries $(1,\lceil n/2\rceil +1)$ and
$(\lfloor n/2\rfloor ,\lceil n/2\rceil +1)$ are red, and the entries
$(2,n)$ and $(\lceil n/2\rceil ,\lceil n/2\rceil +2)$ are blue.
\end{enumerate}
\end{theorem}

\begin{figure}
[h]
\begin{center}
\includegraphics[
width=6.5in ] {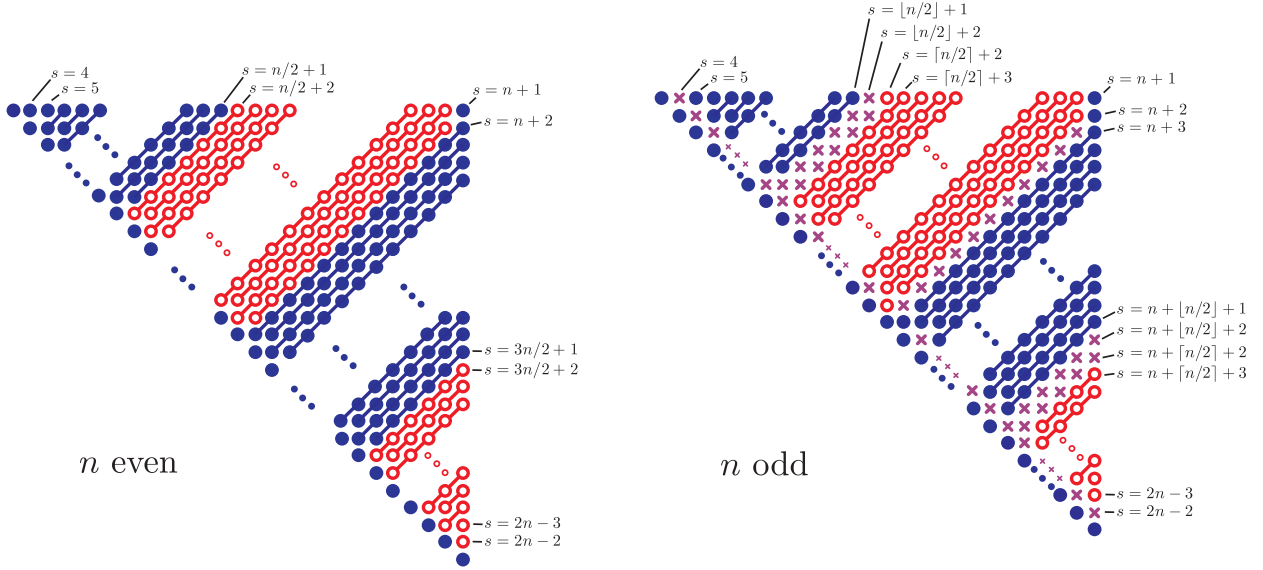} \caption{The even and odd cases in
Theorem \ref{mainstructure}. The crosses in
the odd case represent points whose color is not fixed.}%
\label{evenandodd}
\end{center}
\end{figure}

\begin{proof}
Let
\begin{align*}
T_{U}(D)  & =\{(r,c)\in M(D):2\leq c\leq\left\lceil n/2\right\rceil ,1\leq
r\leq c-1\},\\
R(D)  & =\{(r,c)\in M(D):\left\lceil n/2\right\rceil +1\leq c\leq n,1\leq
r\leq\left\lceil n/2\right\rceil \},\text{ and}\\
T_{L}(D)  & =\{(r,c)\in M(D):\left\lceil n/2\right\rceil +1\leq c\leq
n,\left\lceil n/2\right\rceil +1\leq r\leq c-1\}.
\end{align*}

We shall prove the theorem first for those entries that lie on
$R(D)$, then for those that lie on $T_U(D)$, and finally for those
that lie on $T_L(D)$.
\vglue 0.2 cm
\noindent{\em The entries in
$R(D)$}
\vglue 0.2 cm

We refer the reader to
 Figure \ref{equality01}. By Lemma
\ref{start}, we can assume that in $M(D)  $
\begin{equation}
\text{the entries }(1,n),(2,n),\ldots,(\lfloor n/2 \rfloor  ,n)\text{
are blue}\label{initialcolumn}%
\end{equation}
(in fact $(\lceil n/2 \rceil,n)$ can also be assumed to be blue but we do not use this fact) and%
\begin{equation}
\text{the entries }(1,\left\lceil n/2\right\rceil +1),\ldots,(1,n-1)\text{ are
red.}\label{initialrow}%
\end{equation}
Moreover, we can assume that
\begin{equation}
\text{the entry }(2,n-1)\text{ is red.}\label{initialextrapoint}%
\end{equation}
(If it is blue, then $M(h\circ g(D))$ satisfies
(\ref{initialcolumn}), (\ref{initialrow}), and
(\ref{initialextrapoint})).
\begin{figure}
[ptb]
\begin{center}
\includegraphics[
height=2.4777in,
width=4.8153in]
{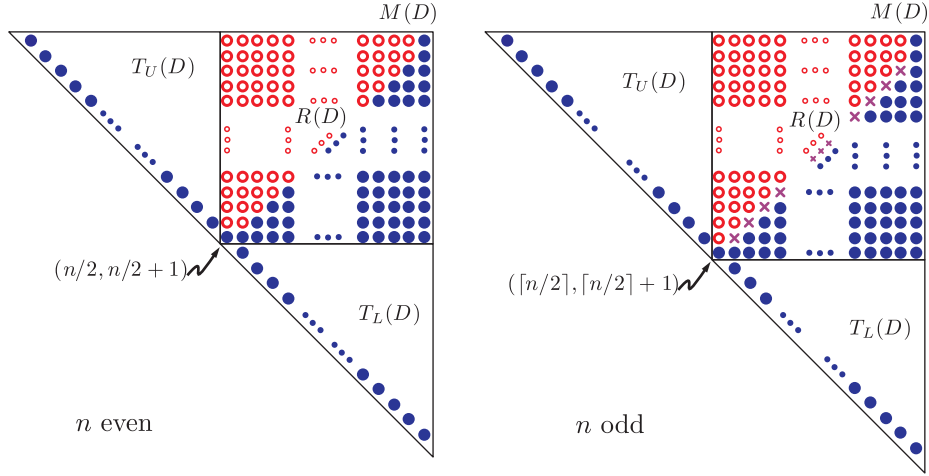}
\caption{The regions $T_{U}(D),R(D)$, and $T_{L}(D)$.}
\label{equality01}
\end{center}
\end{figure}

We now prove that for each $r$ such that $2\leq
r\leq \lfloor{n/2}\rfloor$,
\begin{equation}
\text{the entries }(  r,\lceil n/2\rceil +1) ,( r,\lceil n/2\rceil
+2)  ,\ldots,( r,2\lfloor n/2\rfloor -r+1)  \text{ are
red}\label{SquareUpper}
\end{equation}
and
\begin{equation}
\text{the entries }\left(  r,2\left\lceil n/2\right\rceil
-r+2\right) ,\left(  r,2\left\lceil n/2\right\rceil -r+3\right)
,\ldots,(r,n)\text{ are blue.}\label{SquareLower}
\end{equation}


Observe that if $r=2$ and $n$ is even, then
\eqref{SquareLower} only concerns the entry $(2,n)$, which is
  blue by \eqref{initialcolumn}. (For $r=2$ and $n$  odd,
\eqref{SquareLower} is an empty claim.)
Thus we only need to take care of the
  base case $r=2$ for \eqref{SquareUpper}.  Since (by
\eqref{initialextrapoint}) the entry $( 2,n-1) $ is red, by
Corollary \ref{OneAbove} the first $\lfloor n/2\rfloor -2$ entries
in the order associated to $( 2,n-1) $ have a blue point above. By
(\ref{initialrow}) the only candidates to have blue points above
them are the $\lceil n/2\rceil -2$ entries $( 2,3) ,( 2,4) ,\ldots,(
2,\lceil n/2\rceil ) $. (Note that the order associated to the entry
$( i,j) $ only applies to entries in row $j$ to the left of entry $(
i,j) $.) Thus the $\lceil n/2\rceil -2$ entries $( 1,3) ,( 1,4)
,\ldots,( 1,\lceil n/2\rceil ) $ are blue if $n$ is even, and at
most one of them, say $(1,c_{1})$, is red if $n$ is odd. Moreover,
by Lemma \ref{OrderObservation} the entries $( 2,\lceil n/2\rceil
+1) ,(2,\lceil n/2\rceil +2),\ldots,( 2,n-2) $ are red.

For the inductive step, suppose that for some $3\leq t\leq\lfloor
n/2\rfloor $, each row $r$ with $2\leq r\leq t-1$ satisfies the
result. We now prove (\ref{SquareUpper}) and (\ref{SquareLower}) for
$r=t$. Suppose that the entry $( t,2\lceil n/2\rceil -t+2) $ is red.
Then by Corollary \ref{OneAbove} each of the first $\lceil n/2\rceil
-t+1$ entries in the order associated to $( t,2\lceil n/2\rceil
-t+2)  $ has at least one blue entry above. Since the entries
$(t,\lceil n/2\rceil +1),\ldots,(t,2\lfloor n/2\rfloor -t+2)$ have
all red above, the only candidates are the $\lceil n/2\rceil -t$
entries $(t,t+1),(t,t+2),\ldots,(t,\lceil n/2\rceil )$ and the entry
$2\lfloor n/2\rfloor -t+3=2\lceil n/2\rceil -t+1$ for odd $n$. But,
by Lemma \ref{OrderObservation}, to be a candidate this last entry
should be blue, which is impossible because it would be the first
entry in the order associated to $(t,2\lceil n/2\rceil -t+2)$ with
at most one blue entry above, contradicting Lemma \ref{order}. Since
there are not enough candidates, then the entry $( t,2\lceil
n/2\rceil -t+2)  $ is blue.

Now consider the blue entry $(  t,n)  $. By Corollary \ref{AllAbove}
the first $\lfloor n/2\rfloor -t+1$ entries in the order associated
to $(  t,n)  $ have all entries above them red. The only candidates
are $(t,c_{1})$ if it exists, $( t,\lceil n/2\rceil +1) ,\ldots,(
t,2\lceil n/2\rceil -t+1)  $. For $n$ even, there are $\lceil
n/2\rceil -t+1=\lfloor n/2\rfloor -t+1$ candidates because
$(t,c_{1})$ does not exists, and thus all of them are red by Lemma
\ref{OrderObservation}. For $n$ odd, there are at most 2 more
candidates than we need. By Lemma \ref{OrderObservation} any blue
entry $(t,c)$ with $c\geq\lfloor n/2\rfloor +2$ is not a candidate.
Thus at most two of the last $\lceil n/2\rceil -t+1$ candidates are
blue. Suppose that one of the entries $( t,\lceil n/2\rceil +1) ,(
t,\lceil n/2\rceil +2) ,\ldots,( t,2\lfloor n/2\rfloor -t+1)  $ is
blue. Then there exists $\lceil n/2\rceil +1\leq c\leq2\lceil
n/2\rceil -t$ such that $( t,c) $ is blue and $(t,c+1)$ is red. Then
$(t,c)$ is the first entry in the order associated to $(t,c+1)$ and
all entries above it are red, contradicting Corollary
\ref{OneAbove}. Thus (\ref{SquareUpper}) holds and, by Lemma
\ref{OrderObservation} for $(i,j)=(t,n)$, the rest of
(\ref{SquareLower}) holds too.

Note that \eqref{SquareUpper} is vacuous if
  $r=\lfloor{n/2}\rfloor$ and $n$ is odd. On the other hand, we argue that it is possible to assume that
\begin{equation}
\text{for odd $n$, the entry }(\lfloor n/2 \rfloor, \lceil n/2
\rceil +1)\text{ is red.} \label{extradiagonal2}
\end{equation}

Indeed, suppose that it is blue. Then, by Lemma \ref{halving},
$(\lfloor n/2\rfloor ,\lceil n/2\rceil +1)$ is a blue halving entry
with $\lfloor n/2\rfloor -1$ red entries above and thus all $\lfloor
n/2\rfloor -1$ entries to its right are blue. Hence, by Lemma
\ref{halving}, $(\lfloor n/2\rfloor ,\lceil n/2\rceil )$ is halving
with $\lfloor n/2\rfloor $ blue entries to its right and thus all
$\lfloor n/2\rfloor -1$ entries above are red. Note that
$M(f^{\lfloor n/2\rfloor }(D))$ satisfies (\ref{initialcolumn}),
(\ref{initialrow}), and (\ref{initialextrapoint}) and its entry
$(\lfloor n/2\rfloor ,\lceil n/2\rceil +1)$ is red. Then we start
with $f^{\lfloor n/2\rfloor }(D)$ instead of $D$.

We now prove that the version of \eqref{SquareLower} for
$r=\lceil{n/2}\rceil$ also holds:
\begin{equation}
\text{the entries } \left(  \lceil{n/2}\rceil,\left\lceil n/2\right\rceil +2\right)
,\left(  \lceil{n/2}\rceil,\left\lceil n/2\right\rceil +3\right)  ,\ldots,(\lceil{n/2}\rceil,n)\text{ are
blue.}\label{SquareLowerExtra}%
\end{equation}

Note that \eqref{SquareLowerExtra} only needs to be proved for odd
$n$, since for even $n$  this is the case $r=\lfloor{n/2}\rfloor$
in \eqref{SquareLower}.
Using \eqref{initialcolumn} and
\eqref{SquareUpper} it follows that all
the entries above $(\lceil{n/2}\rceil,\lceil{n/2}\rceil+1)$ are
red. By Lemma~\ref{halving}
$(\lceil{n/2}\rceil,\lceil{n/2}\rceil+1)$ is a
  halving entry, and so it follows that all the entries to its right
  are blue. This proves (\ref{SquareLowerExtra}).

We now prove that for $2\leq r\leq\lfloor n/2\rfloor -1 $
\begin{equation}
\text{for odd $n$, the entry }(r,n-r+1)\text{ is red.}\label{extradiagonal}%
\end{equation}
Note that \eqref{extradiagonal2} is a version of
\eqref{extradiagonal} for $r=\lfloor n/2 \rfloor$. Observe that
$M(f^{\lceil n/2\rceil }(D))$ satisfies (\ref{initialcolumn}) and
(\ref{initialrow}). If $(2,n-1)$ is red in $M(f^{\lceil n/2\rceil
}(D))$, then the diagonal $(r,n-r)$ with $1\leq r\leq\lfloor
n/2\rfloor -1$ in $M(f^{\lceil n/2\rceil }(D))$ is red by
(\ref{SquareUpper}). This corresponds to the diagonal $(r,n-r+1)$
with $2\leq r\leq\lfloor n/2\rfloor $ in $M(D)$. So now assume that
the entry $(2,n-1)$ is blue in $M(f^{\lceil n/2\rceil }(D))$, which
corresponds to $(\lfloor n/2\rfloor ,\lceil n/2\rceil +2)$ being
blue in $M(D)$. In this case, we can assume that $(1,\lceil
n/2\rceil ) $ is blue. (Otherwise start with $M(h\circ g\circ
f^{\lceil n/2\rceil }(D))$ instead of $D$, which satisfies
(\ref{initialcolumn}), (\ref{initialrow}),
(\ref{initialextrapoint}), $(\lfloor n/2\rfloor ,\lceil n/2\rceil
+1)$ is red, and $(1,\lceil n/2\rceil )$ is blue.) Now, by Lemma
\ref{halving}, $(\lfloor n/2\rfloor ,\lfloor n/2\rfloor +1)$ is a
halving entry with $\lfloor n/2\rfloor $ of the entries in
(\ref{halvingMup}) blue, then all others must be red, i.e.,
$(2,\lceil n/2\rceil ),(3,\lceil n/2\rceil ),\ldots,(\lfloor
n/2\rfloor -1,\lceil n/2\rceil )$ are red. Assume by contradiction
that $(r,n-r+1)$ is blue for some $2\leq r\leq\lfloor n/2\rfloor
-1$. Then $(r,n-r+2)$ is blue, otherwise $(r,n-r+1)$ would be the
first entry in the order associated to $(r,n-r+2)$ with no blue
entry above, contradicting Corollary \ref{OneAbove}. But now the red
entry $(r,\lceil n/2\rceil )$ is the $(\lceil n/2\rceil -r)^{th}$
entry in the order associated to the blue entry $(r,n)$ with a blue
entry above, contradicting Corollary \ref{AllAbove} and proving
\eqref{extradiagonal}.

\begin{figure}
[ptb]
\begin{center}
\includegraphics[
height=2.4751in,
width=4.8862in]
{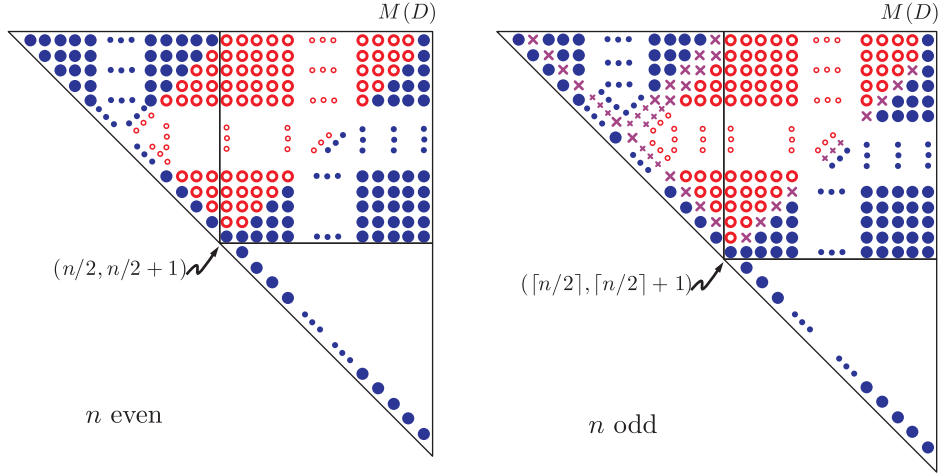}
\caption{The upper triangle $T_{U}(D)$ for even and odd $n$ in the proof of
Theorem \ref{mainstructure}.}
\label{equality02}
\end{center}
\end{figure}

We finally observe that \eqref{initialcolumn}, \eqref{initialrow},
\eqref{initialextrapoint}, \eqref{SquareUpper}, \eqref{SquareLower},
 \eqref{extradiagonal2}, \eqref{SquareLowerExtra}, and
\eqref{extradiagonal} prove Theorem~\ref{mainstructure} for the
entries in $R(D)$.

\vglue 0.2 cm
\noindent{\em The entries in $T_{U}(D)$}
\vglue 0.2 cm

We refer the reader to Figure~\ref{equality02}. We prove by
induction on $c$ that for $1\leq c\leq\lfloor \frac{1}{2} \lceil n/2
\rceil \rfloor $,
\begin{equation}
\text{the entries }(  c+e,\lceil n/2\rceil +2-c) ,\ldots,( \lfloor
n/2\rfloor -c,\lceil n/2\rceil
+2-c)  \text{ are red,}\label{columnred}%
\end{equation}
and%
\begin{equation}
\text{the entries }(  1,\lceil n/2\rceil +2-c) ,\ldots,(
c-1-e,\lceil n/2\rceil +2-c) \text{ are
blue.}\label{columnblue}%
\end{equation}

We have proved it for $c=1$. Suppose that the result holds for all
$1\leq c\leq d-1$ and we now prove it for $c=d$. By Lemma
\ref{halvingMobius} for $i=\lceil n/2\rceil +1-d$, and since by
(\ref{SquareLower}) the $\lfloor n/2\rfloor -d$ entries
$\{(i,b)\mid2\lceil n/2\rceil -i+2\leq b\leq i+\lceil n/2\rceil
\}\cup \{(i+1,b)\mid i+\lceil n/2\rceil +1\leq b\leq n\}$ in
(\ref{halvingMup}) are blue, then $(i,i+1)$ has at most $d-1+e$ blue
entries above. Suppose by contradiction that $(r,i+1)$ is blue for
some $d+e\leq r\leq\lfloor n/2\rfloor -d$. Then $(r,i+1)$ is the
first entry in the order associated to $(r,n-r+1)$ and has at most
$\lceil n/2\rceil -1-(\lfloor n/2\rfloor -d)-1=d-2+e$ blue entries
above. By Lemma \ref{order}, $(r,i+1)$ has at least $\min\{\lfloor
n/2\rfloor -r,r-1\}$ blue entries above and thus $\min\{\lfloor
n/2\rfloor -r,r-1\}\leq d-2+e$. But $r-1>d-2+e$ because $r\geq d+e$,
and $\lfloor n/2\rfloor -r\geq d>d-2+e$ because $r\leq\lfloor
n/2\rfloor -d$. Thus (\ref{columnred}) holds for $c=d$.

Look at $(i,i+1)$ again. The $\lfloor n/2\rfloor -1-3e$ entries
$\{(r,i+1)\mid d+e\leq r\leq i-1-e\}\cup\{(i,b)\mid i+2+e\leq b\leq
n-i+1\}$ in (\ref{halvingMup}) are red and thus, by Lemma
\ref{halvingMobius}, at most other $4e$ entries are red. For $n$
even, $4e=0$ and thus (\ref{columnblue}) holds. For $n$ odd, suppose
by contradiction that $(d-e,i+1)$ has a red entry above. We prove
that in this case the entries $(d-e,i+1),(d-e+1,i+1),$ and
$(i-1,i+1)$ are red. Since $(d-e,n+1-d+e)$ is red, then by Corollary
\ref{AllAbove} the first $\lfloor n/2\rfloor +2-2d+2e$ entries in
the order associated to $(d-e,n+1-d+e)$ have only blue entries
above. If $(d-e,i+1)$ were blue, then it would be one of the first
two entries in the order associated to $(d-e,n+1-d+e)$ with at least
one red point above. This means that $1\geq\lfloor n/2\rfloor
+2-2d+2e$ contradicting that $d\leq\lfloor \frac{1}{2} \lceil n/2
\rceil \rfloor $. Thus $(d-e,i+1)$ is red. Similarly, $(d-e+1,i+1)$
cannot be blue as it would be the first entry in the order
associated to $(d-e+1,n-d+e)$, which by Lemma \ref{order} should
have at most one red entry above, but $(d-e+1,i+1)$ has now at least
2 red entries above. Now $(i-1,i+1)$ is the first entry for
$(i-1,n+2-i)$ and, by (\ref{columnred}), it has at least $\lceil
n/2\rceil +1-2d+e$ red entries above, i.e., at most $d-2-e$ blue
entries above. But by Lemma \ref{order}, the first entry in the
order associated to the red entry $(i-1,n+2-i)$ has at least
$\min\{d-1,i-2\}$ blue entries above. Thus $\min\{d-1,i-2\}\leq
d-2-e$, but $d-1>d-2-e$ and $i-2>d-2-e$ because $d\leq\lfloor
\frac{1}{2} \lceil n/2  \rceil \rfloor $, getting a contradiction.
Hence $(i-1,i+1)$ is red. By Lemma \ref{halvingMobius} at most
$\lfloor n/2\rfloor $ of the entries in (\ref{halvingUpseconddiag})
are red, yet we already have $\lceil n/2\rceil $ red entries
(namely, at least the $\lceil n/2\rceil +1-2d+e$ above $(i-1,i+1)$
mentioned before and the $2d-2$ entries $\{(i-1,b)\mid i+2\leq b\leq
n-i+2\}$ to its right), getting a contradiction. Thus
(\ref{columnblue}) holds for $c=d$.

Now we prove that for $2\leq c\leq\lceil \frac{1}{2}\lceil n/2\rceil
\rceil +1$,
\begin{equation}
\text{the entries }(1,c),(2,c),\ldots,(c-2-e,c)\text{ are blue.}%
\label{halfbluetriang}%
\end{equation}
Since $(c-1,c)$ is one of the $\lfloor n/2\rfloor +5-2c$ entries in
the order associated to the red entry $(c-1,n+2-c)$ (we have shown
that the $\lfloor n/2\rfloor -1-e$ entries immediately to the left
of $(n+2-c)$ are red), then $(c-1,c)$ has at most one red entry
above by Lemma \ref{order}. Suppose by contradiction that $(r,c)$ is
red for some $1\leq r\leq c-2-e$. Then $(r+1,c)$ is blue. Since
$(r+1,n-r)$ is red, then by Corollary \ref{AllAbove} the first
$\lfloor n/2\rfloor -2r$ entries in the order associated to
$(r+1,n-r)$ have only blue entries above. But $(r+1,c)$ is one of
the first $\lfloor n/2\rfloor -2r$ entries and
has the red entry $(r,c)$ above, getting a contradiction.%

We finally note that \eqref{columnred}, \eqref{columnblue}, and
\eqref{halfbluetriang} prove Theorem~\ref{mainstructure} for the
entries in $T_{U}(D)$.

\vglue 0.2 cm \noindent{\em The entries in $T_{L}(D)$} \vglue 0.2 cm
\begin{figure}
[ptb]
\begin{center}
\includegraphics[
height=3.3036in, width=5.9248in ] {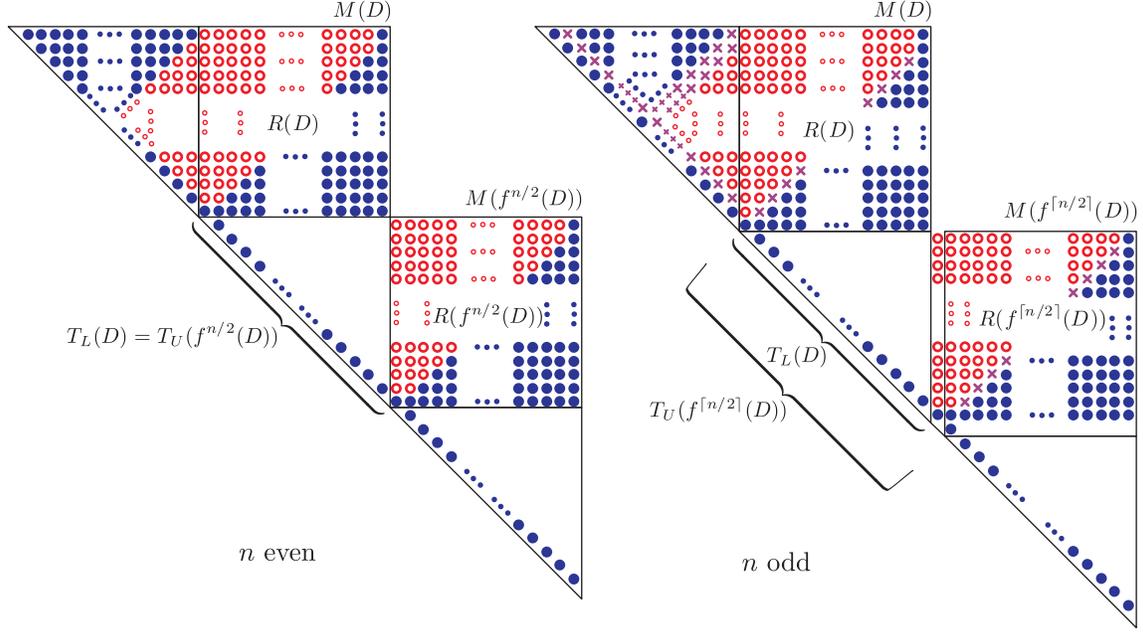} \caption{The
lower triangle $T_{L}(D)$ versus the upper triangle
$T_{U}(f^{\left\lceil n/2\right\rceil }(D))$ for even and odd $n$ in
the proof of Theorem \ref{mainstructure}.}
\label{equality03}%
\end{center}
\end{figure}

We refer the reader to Figure \ref{equality03}. Consider $f^{\lceil
n/2\rceil }(D)$. When $n$ is even, see Figure \ref{equality03}
(left), $R(D)$ and $R(f^{\lceil n/2\rceil }(D))$ are identical and
thus our previous arguments show that $T_{U}(D)$ and
$T_{U}(f^{\lceil n/2\rceil }(D))=T_{L}(D)$ are identical too,
concluding the proof in this case. When $n$ is odd, see Figure
\ref{equality03} (right), $R(D)$ and $R(f^{\lceil n/2\rceil }(D))$
are slightly different:\ for $2 \leq r \leq \lfloor n/2 \rfloor$ the
diagonal entries $(r,n+1-r)$
are red in $R(D)$ and unfixed in $R(f^{\lceil n/2\rceil }(D))$, and
for $3 \leq r \leq \lfloor n/2 \rfloor$ the diagonal entries
$(r,n+2-r)$
are unfixed in $R(D)$ and blue in $R(f^{\lceil n/2\rceil }(D))$.
Also the last row of $R(D)$ is blue and the last row of $R(f^{\lceil
n/2\rceil }(D))$ is unfixed. However, the last column of
$T_{U}(f^{\lceil n/2\rceil }(D))$ is red and this is what allows us
to mimic the arguments used for (\ref{columnred}),
(\ref{columnblue}), and (\ref{halfbluetriang}) to show that
$T_{L}(D)$, which corresponds to $T_{U}(f^{\lceil n/2\rceil }(D))$
minus its last column, satisfies the statement. More precisely, it
can be proved by induction on $c$ that for $1\leq c\leq\lfloor
\frac{1}{2} \lceil n/2  \rceil \rfloor$, in $M(f^{\lceil n/2\rceil
}(D))$
\begin{equation}
\text{the entries }(  c+1,\lceil n/2\rceil +1-c) ,\ldots,( \lfloor
n/2\rfloor -c-1,\lceil n/2\rceil
+1-c)  \text{ are red}%
\label{columnred2}
\end{equation}
and%
\begin{equation}
\text{the entries }(  1,\lceil n/2\rceil +1-c) ,\ldots,( c-2,\lceil
n/2\rceil +1-c)  \text{
  are blue.}
\label{columnblue2}
\end{equation}

We omit the proofs of \eqref{columnred2} and \eqref{columnblue2}, as
they very closely resemble the proofs of \eqref{columnred}
and \eqref{columnblue}.

Similarly, it can be proved by induction that for $2\leq c\leq\lceil
\frac{1}{2} \lceil n/2\rceil \rceil $, in  $M(f^{\lceil
n/2\rceil }(D))$%
\begin{equation}
\text{the entries }(1,c),(2,c),\ldots,(c-3,c)\text{ are blue.}%
\label{halfbluetriang2}
\end{equation}

The proof of \eqref{halfbluetriang2} is also omitted, as it
very closely resembles the proof of \eqref{halfbluetriang}.

We finally note that \eqref{columnred2}, \eqref{columnblue2}, and
\eqref{halfbluetriang2} prove Theorem~\ref{mainstructure} for the
entries in $T_{U}(L)$.
\end{proof}

\subsection{The number of crossing optimal drawings}

Theorem \ref{mainstructure} completely determines $M(D^{\prime})$ when $n$ is
even, which means that in this case there is essentially only one crossing optimal drawing.

\begin{theorem}
\label{OptimalEven}For $n$ even, up to homeomorphism, there
is a unique crossing optimal 2-page book drawing of $K_{n}$.
\end{theorem}

\begin{proof}
The result is easily seen to hold  for $n=2$ and $n=4$. For $n\geq6$
Theorem \ref{mainstructure} completely determines $M(D^{\prime})$.
Note that this matrix corresponds to the drawings by Bla\v{z}ek and
Koman\ \cite{BC64}.
\end{proof}

In contrast to the even case, for $n$ odd there is an exponential
number of non-equivalent crossing optimal 2-page book drawings of
$K_{n}.$ For any odd integer $n\geq5$, we construct $2^{(n-5)/2}$
non-equivalent crossing optimal drawings of $K_{n}$. In fact,
these $2^{(n-5)/2}$ drawings are pairwise non-homeomorphic. To prove
this, we need the next two results.

\begin{theorem}
For every $n\geq13$ odd, every crossing optimal 2-page book drawing of $K_{n}$
has exactly one Hamiltonian cycle of non-crossed edges, namely the one
obtained from the edges on the spine and the $1n$ edge.
\end{theorem}

\begin{proof}
Assume $n\geq13$ is odd. To show that $123\ldots n$ is the only
non-crossed Hamiltonian cycle, we show that all other edges are
crossed at least once. Assume that $D$ has the form described in
Theorem \ref{mainstructure}. Let $(r,c)$ be an entry of $M(D)$ that
does not represent an edge on the spine or the $1n$ edge. Let
\[
(r,c)^{+}=\left\{
\begin{array}
[c]{ll}%
(r+1,c+1) & \text{if }c<n\text{, or}\\
(1,r+1), & \text{if }c=n,
\end{array}
\right.  \text{ and }(r,c)^{-}=\left\{
\begin{array}
[c]{ll}%
(r-1,c-1) & \text{if }r>1\text{, or}\\
(c-1,n), & \text{if }r=1\text{.}%
\end{array}
\right.
\]
Note that the edges corresponding to $(r,c)^{+}$ and $(r,c)^{-}$
cross the edge $rc$ if they have the same color as $(r,c)$.

First assume that $3\leq c-r\leq n-3$. Suppose that $(r,c)$ is a
blue entry specified
 by Theorem \ref{mainstructure}. If $5\leq r+c\leq\lfloor
n/2\rfloor-1$ or if $n+3\leq r+c\leq n+\lfloor n/2\rfloor-1$, then
note that the entry $(r,c)^{+}$ is also blue according to Theorem
\ref{mainstructure}, and thus the edges corresponding to $(r,c)$ and
$(r,c)^{+}$ cross each other.

Because $n\geq13$, if $\lfloor n/2\rfloor\leq r+c\leq\lfloor
n/2\rfloor+1 $ or $n+\lfloor n/2\rfloor\leq r+c\leq n+\lfloor
n/2\rfloor+1$, then $5\leq\lfloor n/2\rfloor-2\leq r+c-2\leq\lfloor
n/2\rfloor+1$ or $n+3\leq n+\lfloor n/2\rfloor-2\leq r+c-2\leq
n+\lfloor n/2\rfloor+1$, respectively. Thus the entry $(r,c)^{-}$ is
also blue according to Theorem \ref{mainstructure}, and thus the
edges corresponding to $(r,c)$ and $(r,c)^{-}$ cross each other.

A similar argument shows that for every red entry $(r,c)$ specified
by Theorem \ref{mainstructure}, either $(r,c)^{+}$ or $(r,c)^{-}$ is
also a red edge.

Second, assume that $c-r=n-2$, that is $(r,c)\in\{(1,n-1),(2,n)\}$.
If $(r,c)=(1,n-1)$, then $(r,c)$ is red and because $(2n-4)\geq
n+\lceil n/2\rceil+2$ for $n\geq13$, it follows that $rc$ crosses
the edge corresponding to $(n-3,n)$, which is red. If $(r,c)=(2,n)$,
then $(r,c)$ is blue and because $\lfloor n/2\rfloor\geq4$ for
$n\geq13$, it follows that $rc$ crosses the edge corresponding to
$(1,4)$, which is blue.

Suppose now that the color of $(r,c)$ is not determined by Theorem
\ref{mainstructure}. First assume that $r+c\in\{\lfloor
n/2\rfloor+2,\lceil n/2\rceil+2,n+\lfloor n/2\rfloor+2,n+\lceil
n/2\rceil+2\}$. Again, by Theorem \ref{mainstructure} note that
$(r,c)^{-}$ is blue and $(r,c)^{+}$ is red. Similarly, if
$r+c=n+2$, then $(r,c)^-$ is red and $(r,c)^+$ is blue. Thus
regardless of its color, the edge $rc$ will cross one of the two
edges corresponding to these two entries.

Finally, assume $c-r=2$. From Theorem \ref{mainstructure}, the
number of red entries of the form $(t,r+1)$ or $(r+1,d)$, with
$1\leq t\leq r$ and $r+3\leq d\leq n$ is at least $\lfloor
n/2\rfloor -5\geq1$. A similar statement holds for the number of
blue entries of the same form. Thus there is at least one blue edge
(not on the spine) and at least one red edge incident to $r+1$. One
of these two edges will necessarily cross the edge $rc$ regardless
of its color.
\end{proof}

Note that for $n \leq11$ the above approach cannot guarantee that
there are no additional non-crossed edges. For example for $n=11$ the
element $(1,10)$ cannot be determined. However, these small
cases can be handled by exhaustive enumeration, which shows that for
crossing optimal drawings there are no such edges for $n=11$ and no
alternative Hamiltonian cycles for $n=9$. For $n=5,7$ there exist
alternative Hamiltonian cycles of non-crossed edges, but they do not
lead to additional equivalences between the crossing optimal drawings.

\begin{corollary}
\label{Cor:homeomorphic}If $D$ and $D^{\prime}$ are crossing optimal 2-page
book drawings of $K_{n}$, then either $D$ and $D^{\prime}$ are not
homeomorphic, or else $M(D)$ and $M(D^{\prime})$ are equivalent.
\end{corollary}

\begin{proof}
If $n$ is even the result is trivial by Theorem \ref{OptimalEven}.
If $n$ is odd and $n\leq11$, then using Theorem \ref{mainstructure}
we exhaustively found all equivalence classes of crossing optimal
drawings. There are $1$, $4$, $9$, and $25$ equivalence classes for
$n=5$, 7, 9, and 11, respectively. We verified that all of these
equivalence classes were topologically distinct. If $n\geq13$ and
$D$ and $D^{\prime}$ are crossing optimal 2-page book drawings, then
by the previous theorem both $D$ and $D^{\prime}$ have only one
non-crossed Hamiltonian cycle. Thus if $H:D\rightarrow D^{\prime}$
is a homeomorphism, then $H$ must send the Hamiltonian cycle
$123\ldots n$ to itself. It follows that $H$ restricted to this
cycle is the composition of a rotation of the cycle with either the
identity, or the function that reverses the order of the cycle.
Moreover, once the edges on the spine are fixed, the drawing is
determined by the colors of the remaining edges. Thus either $H$ is
determined by its action on the cycle, or else $H$ switches the blue
edges not on the spine with the red edges. In other words,
$M(D^{\prime})=M(H(D))=M((h^{a}\circ g^{b}\circ f^{i})(D))$ for some
$i\in\{0,1,2,\ldots,n-1\}$ and $a,b\in\{0,1\}$. Thus $M(D)$ and
$M(D^{\prime})$ are equivalent.
\end{proof}

\begin{theorem}
\label{OptimalOdd}For $n$ odd, there are at least $2^{(n-5)/2}$
 pairwise non-homeomorphic crossing optimal 2-page book
drawings of $K_{n}$.
\end{theorem}

\begin{proof}
As usual let $1,2,\ldots,n$ be the vertices of $K_{n}$. Let $rc$ be
an edge of $K_{n}$ that is not on the Hamiltonian cycle $H=$
$12\ldots n$, we color $rc$ red or blue according to the following
rule: if $r+c\equiv s\pmod n$ for some integer $2\leq s\leq(n+1)/2$,
then we color $rc$ blue, if $r+c\equiv s\pmod n$ for some integer
$(n+5)/2\leq s\leq n+1$, then we color $rc$ red. Finally, if
$r+c\equiv(n+3)/2\pmod n$, then we color $rc$ either red or blue.
See (Figure \ref{construction}.)

We first argue that all of these colorings yield crossing optimal
drawings of $K_{n}$ regardless of the color of the $(n-3)/2$ edges
$rc$ for which $r+c\equiv(n+3)/2\pmod n$.

\begin{figure}
[ptb]
\begin{center}
\includegraphics[
scale=1]
{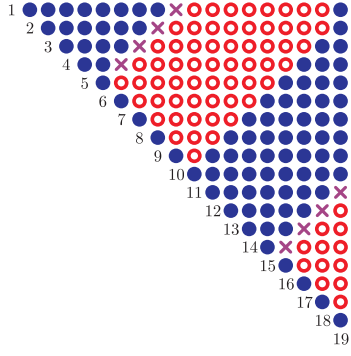}
\caption{The $2^8$ crossing optimal drawings (only $2^7$ non-equivalent) for $n=19$ in Theorem~\ref{OptimalOdd}. They are obtained by assigning arbitrary colors to the crosses in this matrix.}
\label{construction}
\end{center}
\end{figure}

For every $1\leq s\leq n$, let $I_{s}=\{rc$ edge$:rc\notin H$ and
$r+c\equiv
s\pmod n\}$. Note that $|I_{s}|=(n-3)/2$ for all $s$ and $\bigcup_{s=1}%
^{n}I_{s}$ is the complete set of edges not in $H$. Moreover note
that each $I_{s}$ is a matching of pairwise non-crossing edges.

Let $rc$ be an edge such that $r+c\equiv(n+3)/2\pmod n$. Assume
without loss of generality that $r<c$. \ If \ $td$ is an edge that
crosses $rc$, then $t$ and $d$ are cyclically separated from $r$ and
$c$; that is, we may assume that $r<t<c$ and $d<r$ or $d>c$. To
facilitate the case analysis we may assume that the edges that could
cross $rc$ are the edges $td$ such that $r<t<c<d<n+r$, with the
understanding that $d$ represents the point $d-n$ when $d>n$. Let
$C=\{td$ edge$:r<t<c<d<n+r\}$ and
consider the function $T:C\rightarrow C$ defined by $T(td)=t^{\prime}%
d^{\prime}$ where $t^{\prime}=r+c-t$ and $d^{\prime}=r+c+n-d$. Note
that $T$ is well defined because $r<t^{\prime}<c<d^{\prime}<n+r$ and
$T$ is one-to-one
on $C$. 
Moreover, note that
\begin{align*}
t^{\prime}+d^{\prime}  & \equiv r+c+n+r+c-t-d\pmod n\\
& \equiv2(r+c)-(t+d)\pmod n\\
& \equiv(n+3)-(t+d)\equiv 3-(t+d)\pmod n,
\end{align*}
so $t+d\equiv s\pmod n$ with $2\leq s\leq(n+1)/2$ if and only if
$t^{\prime }+d^{\prime}\equiv 3-(t+d)\equiv n+ 3 -s\pmod n$ and
$(n+5)/2\leq n+3-s\leq n+1$. Thus $td$ and $T(td)$ have different
colors, which means that $C$ contains as many red edges as blue
edges. Hence $rc$ crosses the same number of edges independently of
its color. This shows that all the drawings we have described have
the same number of crossings. Finally, we note that the drawing for
which all the arbitrary edges have the same color corresponds to the
construction originally found by Bla\v{z}ek and Koman\ \cite{BC64}
having exactly $Z(n)=\frac{1}{64}(n-1)^{2}(n-3)^{2}$ crossings.
Hence all the other drawings described are crossing optimal as well.

We now argue that every drawing constructed here is equivalent to
exactly one other drawing, and thus we have constructed exactly
$2^{(n-5)/2}$ distinct topological drawings. Let $D$ and
$D^{\prime}$ be two of the crossing optimal drawings we just
constructed and suppose that $D$ and $D^{\prime}$ are homeomorphic.
By Corollary \ref{Cor:homeomorphic}, $M(D)$ and $M(D^{\prime})$ are
equivalent, thus there exists a transformation $F:D\rightarrow
D^{\prime}$ such that $F=h^{a}\circ g^{b}\circ f^{i}$ with
$i\in\{0,1,2,\ldots,n-1\}$ and $b,a\in\{0,1\}$. First observe that
under $f$, $g$, or $h$, the absolute value difference of the number
of red minus blue edges remains invariant. Thus the drawing $D$ in
which all of the edges in $I_{(n+3)/2}$ are red can only be
homeomorphic to the drawing $D^{\prime}$ in which all of those edges
are blue. These two are indeed homeomorphic under the function
$F=h\circ g\circ f^{(n+1)/2}$. Now suppose that the edges
$I_{(n+3)/2}$ in $D$ and in $D^{\prime}$ are not all of the same
color. Note that $f,g,$ and $h$ send $I_m$ into another $I_{m'}$,
and if $I_m$ is monochromatic (all edges of $I_m$ have the same
color) in $D$, then $I_{m'}$ is monochromatic in $f(D),g(D),$ and
$h(D)$. Since $I_m$ is monochromatic in $D$ if and only if $m\neq
(n+3)/2$, then $F$ must send $I_{(n+3)/2}$ to itself. If $b=0$,
$rc\in I_{(n+3)/2}$, and $r^{\prime}c^{\prime}$ is the image of $rc$
under $F$, then $r^{\prime}+c^{\prime}\equiv r-i+c-i\pmod n$. Thus
$r^{\prime }+c^{\prime}\equiv r+c\pmod n$ if and only if $i=0$.
Because the edges $I_{1}$ in $D$ are blue and the edges $I_{1}$ in
$h(D)$ are red, it follows that $a=0$ and thus $F$ is the identity.
Last, if $b=1$, $rc\in I_{(n+3)/2}$, and $r^{\prime}c^{\prime}$ is
the image of $rc$ under $F$, then
$r^{\prime}+c^{\prime}\equiv(n+1-(c-i))+(n+1-(r-i))\equiv
2+2i-(r+c)\pmod n$. Thus $r^{\prime}+c^{\prime}\equiv r+c\pmod n$ if
and only if $i=(n+1)/2$. Because the edges $I_{1}$ in both $D$ and
$h(f^{(n+1)/2}(D))$ are blue, it follows that $a=1$ and thus
$F=h\circ g\circ f^{(n+1)/2}$. It can be verified that indeed $F(D)$
is one of the drawings we constructed here, and thus exactly half of
the drawings we described are pairwise non-isomorphic.
\end{proof}

\begin{table}[ht]
\centering
\begin{tabular}{||r|r||r|r||r|r||}
 $n$ & \multicolumn{1}{|c||}{drawings} & $n$ & \multicolumn{1}{|c||}{drawings} & $n$ & \multicolumn{1}{|c||}{drawings}\\
\hline
 5 &       1 &     17 &     324 &     29 &   38944 \\
 7 &       4 &     19 &     748 &     31 &   84064 \\
 9 &       9 &     21 &    1672 &     33 &  180288 \\
11 &      25 &     23 &    3736 &     35 &  385216 \\
13 &      58 &     25 &    8208 &     37 &  819328 \\
15 &     142 &     27 &   17968 & & \\
\end{tabular}

\vspace{3ex}

\caption{The number of non-homeomorphic crossing optimal $2$-page book drawings of~$K_{n}$ for odd $n$, $5 \leq n \leq 37$.}
\label{table_odd}
\end{table}

The above theorem gives a lower bound of $2^{(n-5)/2}$ for the number
of non-equivalent crossing optimal drawings.  As in the crossing
optimal drawings of Theorem~\ref{mainstructure} there are
$\frac{5}{2}(n-5)$ entries with non-fixed colors, we get an upper
bound of $2^{5({n-5})/{2}}$ non-equivalent crossing optimal
drawings. With exhaustive enumeration we have been able to determine
the exact numbers of non-equivalent crossing optimal drawings for
$n\leq 37$, cf. Table~\ref{table_odd}. The obtained results suggest an
asymptotic growth of roughly $2^{0.54 n}$, rather close to our lower
bound.

\section{Concluding remarks}

\label{sec:concludingremarks}

It was proved by \'Abrego and Fern\'andez-Merchant \cite{AF05}
and by Lov\'asz et al.\ \cite{LVWW04} that the inequality $E_{\leq
k}\left( P\right) \geq 3\tbinom{k+2}{2}$ holds (in the geometric
setting) for every set $P$ of $n$ points in general position in the
plane and for every $k$ such that $0\leq k\leq \left\lfloor
n/2\right\rfloor -2$. This inequality used with the rectilinear
version of Theorem~\ref{CrvsEdges} gives $Z\left( n\right) $ as a
lower bound for the rectilinear crossing number of $K_{n}$
\cite{AF05}. In contrast to the rectilinear case, the inequality
$E_{\leq k}\left( D\right) \geq 3 \binom{k+2}{2}$ does not hold in
general for topological drawings $D$ of $K_{n}$, not even for
general 2-page drawings (see Figure~\ref{counterex}). This shows the
relevance of introducing the parameter $E_{\leq \leq k}\left(
D\right)$ (for which Theorem~\ref{ThreeThrees} can be established,
leading to the $2$-page crossing number of $K_n$). However, the
inequality $E_{\leq k}\left( D\right) \geq 3\binom{k+2}{2}$ does
hold for crossing optimal $2$-page drawings of~$K_{n}$. For a proof
of this, and other interesting observations on crossing optimal
drawings of~$K_n$, we refer the reader to Section 4 in the
proceedings version of this paper~\cite{AAFRS}.

Our approach to determine $k$-edges in the topological setting is to define
the orientation of three vertices by the orientation of the corresponding
triangle in a good drawing of the complete graph. It is natural to ask whether
this defines an abstract order type. To this end, the setting would have to
satisfy the axiomatic system described by Knuth~\cite{knuth}. But it is easy
to construct an example which does not fulfill these axioms, that is, our
setting does not constitute an abstract order type as described by
Knuth~\cite{knuth}. It is an interesting question for further research how
this new concept compares to the classic order type, both in terms of theory
(realizability, etc.) and applications.

\begin{figure}
\begin{center}
\includegraphics[width=6in, trim=0cm 0.2in 0cm .2in, clip=true ] {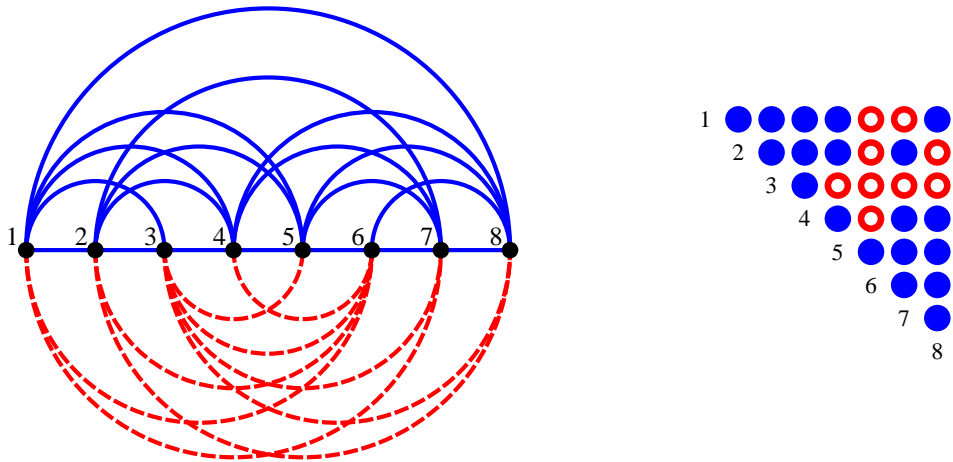}
\caption{A 2-page book drawing of $K_{8}$ with four 0-edges
(namely $(1,7)$, $(1,8)$, $(2,7)$, and $(2,8)$) and  four 1-edges
  (namely $(1,5)$
$(1,6)$, $(3,8)$, and $(4,8)$). This shows that the inequality
$E_{\leq k}( D) \geq 3 \binom{k+2}{2}$, which holds for every
geometric drawing $D$ of $K_n$, does not necessarily hold if $D$ is
a topological drawing. } \label{counterex}
\end{center}
\end{figure}

We believe that the developed techniques of generalized orientation,
$k$-edge for topological drawings, and ${\leq}{\leq}k$-edges are of
interest in their own. We will investigate their usefulness for
related problems in future work. For example, they might also play a
central role to approach the crossing number problem for general
drawings of complete and complete bipartite graphs.

\section{Acknowledgments}

O. Aichholzer is partially supported by the ESF EUROCORES programme
EuroGIGA, CRP ComPoSe, under grant FWF [Austrian Fonds zur
  F\"{o}rderung der Wissenschaftlichen Forschung] I648-N18. P. Ramos
is partially supported by MEC grant MTM2011-22792 and by the ESF
EUROCORES programme EuroGIGA, CRP ComPoSe, under grant
EUI-EURC-2011-4306. G.~Salazar is supported by CONACYT grant
106432. This work was initiated during the workshop Crossing Numbers
Turn Useful, held at the Banff International Research Station
(BIRS). The authors thank the BIRS authorities and staff for their
support.

\end{document}